\documentclass{amsart}
  \usepackage{amscd,amssymb,epsfig, epsf}
   \usepackage{epic,eepic}
 \usepackage{epstopdf}

    \usepackage[all]{xy}
    \SelectTips{cm}{}
    \allowdisplaybreaks

    \numberwithin{equation}{subsection}

    \newtheorem{propo}{Proposition}[section]
    \newtheorem{corol}[propo]{Corollary}
    \newtheorem{theor}[propo]{Theorem}
    \newtheorem{lemma}[propo]{Lemma}

    \theoremstyle{definition}

    \theoremstyle{remark}

\newcommand{\ZZ}{\mathbb{Z}}

\newcommand{\RR}{\mathbb{R}}

 \newcommand{\G}{\mathcal{G}}
 \newcommand{\E}{\mathcal{E}}
 \newcommand{\F}{\mathcal{F}}
 \newcommand{\A}{\mathcal{A}}
  \newcommand{\B}{\mathcal{B}}
 \newcommand{\I}{\mathcal{I}}
\newcommand{\sal}{\mathcal{S}}
  \newcommand{\D}{\mathcal{D}}

\newcommand{\End}{\operatorname{End}}

\newcommand{\Ker}{\operatorname{Ker}}
\newcommand{\Ima}{\operatorname{Im}}

\newcommand{\card}{\operatorname{card}}

\newcommand{\id}{\operatorname{id}}

\newcommand{\rel}{\operatorname{rel}}

\let\oldmarginpar\marginpar
\renewcommand\marginpar[1]{\oldmarginpar{\footnotesize #1}}

\begin{document}

    \title[{Dimer spaces   and gliding systems}]{Dimer spaces   and gliding systems}

    \author[Vladimir Turaev]{Vladimir Turaev}
    \address{
    Vladimir Turaev \newline
    \indent   Department of Mathematics \newline
    \indent  Indiana University \newline
    \indent Bloomington IN47405, USA\newline
    \indent $\mathtt{vturaev@yahoo.com}$}

                     \begin{abstract} Dimer coverings (or perfect matchings) of a finite graph are classical objects of graph theory appearing in   the study of exactly solvable models of statistical mechanics. We introduce more general dimer labelings which form a topological space called the dimer space  of the graph. This  space turns out to be  a cubed complex whose vertices are the dimer coverings. We show that the dimer space is nonpositively curved in the sense  of Gromov, so that its universal covering      is     a CAT(0)-space. We  study the fundamental group of the dimer space and, in particular, obtain  a presentation of this group by generators and relations. We  discuss connections with right-angled Artin groups and   braid groups of graphs. Our approach uses so-called gliding systems in groups designed to produce nonpositively curved cubed complexes.
 \end{abstract}
                     \maketitle

   \section {Introduction}

   A dimer covering, or a perfect matching, of a  graph   is a set of
   edges
   such that each vertex   is incident to exactly one of them.
   Dimer coverings have been extensively studied since 1960's in connection with   exactly solvable models of statistical mechanics and, more recently, in connection with path algebras, see \cite{Bo},   \cite{Ken2} and  references therein. In the present paper we develop a  geometric  approach to dimer coverings. Specifically, we introduce and study  a dimer space and a dimer group of a finite graph.

Consider a finite graph ${\Gamma}$ (without loops but possibly with multiple edges).    A {\it dimer labeling} of   ${\Gamma}$ is a
labeling of the edges of   ${\Gamma}$  by non-negative real numbers such
that for every vertex of   ${\Gamma}$, the   labels of the adjacent
edges
 sum up  to give $1$ and    only one or two of   these labels may be non-zero.
 The set $ L=L({\Gamma})$ of dimer labelings  of   ${\Gamma}$ is a closed subset of  the cube formed by
all labelings of edges   by numbers in  $ [0,1]$. We endow
$L $
  with the induced  topology. We will show that   $L$ is a CW-complex, possibly disconnected, and   all connected  components of   $L  $ are
aspherical, i.e.,   have trivial higher homotopy groups.

The characteristic function of a  dimer covering of $\Gamma$, carrying the edges of
the  covering  to $1$  and  all other edges  to $0$, is a dimer labeling.    The  points of   $L $  represented in this way by
dimer coverings   lie in a single    component
$L_0=L_0({\Gamma})$  of~$L$ called the {\it dimer space}.
 The
fundamental group  of  $L_0$ is   the {\it dimer group} of~${\Gamma}$. We show that the dimer group  is torsion-free,
residually nilpotent, residually finite, biorderable, biautomatic,
    has solvable word and conjugacy
problems, satisfies   the Tits alternative,     embeds in $SL_n(\ZZ)$
for some~$n$, and embeds in   a finitely generated
right-handed Artin group. The fundamental groups of other components of $L$ are isomorphic to the dimer groups of certain subgraphs of ${\Gamma}$ and share the properties listed above.

The reader  familiar with geometric group theory will
immediately recognize   the  properties of groups  arising in the
study of  Cartan-Alexandrov-Toponogov (0)-spaces in the sense of Gromov, briefly called
    CAT(0)-spaces. Such a
space has a geodesic metric in which all geodesic triangles are
at least as thin as   triangles in~$\RR^2$ with the same sides.  The
 theory of CAT(0)-spaces will be our main tool.
We show that   $L_0 $ has a natural structure of
a cubed complex whose vertices are   the dimer coverings of
${\Gamma}$. This cubed complex satisfies Gromov's link
condition, and therefore~$L_0 $ is nonpositively curved.
The universal covering   of $L_0 $  is     a
  CAT(0)-space. Similar results hold for other components of $L$.

The cubed structure in $L_0 $ arises from the following
 observation. Consider
  an embedded circle $s $ in $ {\Gamma}$ of even length (i.e., formed by an even number of edges). Suppose that   every second edge of $s$
belongs to  a dimer covering $A$ of ${\Gamma}$. Removing these edges from $A$   and adding instead  all other edges of $s$  we
obtain a new dimer covering   $sA$. We say that $sA$ is obtained
from $A$ by {\it gliding} along $s$. Deforming the labeling
determined by $A$ into the labeling determined by $sA$ we   obtain a
path (a 1-cube)  in $L_0 $. More generally, we can
start with a  family of $k\geq 1$ disjoint embedded circles in
${\Gamma}$  of even length  meeting $A$ along every second edge. Gliding $A$ along
(some of) these circles, we obtain   $2^k$ dimer coverings of
${\Gamma}$ that serve as the vertices of a $k$-dimensional cube in
$L_0 $. Deforming   the labelings
associated with these  vertices we   obtain all   points of the
cube. Similar ideas were   introduced in \cite{STCR} in the study of domino tilings of planar regions.

Besides the properties of the dimer group  mentioned above, we obtain  a  presentation of this group     by generators and relations. Denote the set of dimer coverings of ${\Gamma}$ by $\D$. For   $A\in \D$ and a vertex $v$ of ${\Gamma}$, denote the only edge of~$A$ incident to $v$ by $A_v$.    We say that a triple of  dimer coverings $A,B,C\in \D$ is {\it flat} if for any vertex~$v$ of ${\Gamma}$, at least two of the edges $A_v, B_v, C_v$ coincide.   For each $A_0\in \D$, the dimer group $\pi_1(L_0, A_0)$ is generated by the symbols
$\{y_{A,B}\}_{A,B \in \D}$ numerated by   ordered pairs of dimer coverings. The defining relations are as follows:   $y_{A,C}=y_{A,B} \, y_{B,C}$ for any flat triple  $A,B,C\in \D$ and $y_{A_0,A}=1$ for all $A\in \D$. We obtain a similar presentation for the fundamental groupoid of the pair $(L_0, \D)$.

Among other results of the paper note  a   connection of the dimer groups to  the braid groups of   graphs and a generalization of the   dimer   groups     to hypergraphs.

The study of dimer coverings    suggests an axiomatic framework of  gliding systems. A {\it gliding
system} in a group $G$ consists of certain elements of   $G$ called  {\it  glides} and a  binary relation on the set of glides called {\it independence} satisfying a few   axioms.   Given a gliding system in $G$ and a  set $\E \subset G$, we construct a cubed complex $X_\E$ called the {\it glide complex}.       The fundamental groups of the components of $X_\E$ are  the {\it glide groups}. We formulate  conditions   ensuring that $X_\E$ is nonpositively curved. One can view gliding systems as  devices producing   nonpositively curved complexes and interesting groups.
  The dimer space    of a graph      is an instance  of the glide  complex  where   $G$ is the     group of   $\ZZ/2\ZZ$-valued function on the set of edges, the  glides are the characteristic functions of the sets of edges forming  embedded circles of even length, the independence mirrors the disjointness of embedded circles, and~$\E  $ consists of the    characteristic functions of   dimer coverings.

The paper is organized as follows. Sections
\ref{sect-prelim}--\ref{lloccgccgcgl} are concerned with glides. We define gliding systems (Section~\ref{sect-prelim}), construct the glide complexes (Section~\ref{ From glides to cubed   complexes}), study   natural   maps between the glide groups (Section~\ref{incmaps}),
embed the glide groups into right-angled Artin groups (Section~\ref{embArtin}), produce  presentations of the glide groups by generators and relations (Section~\ref{redu}),  and  study a   class of set-like gliding systems (Section~\ref{lloccgccgcgl}).   Next, we introduce   and study dimer complexes (Section~\ref{dimergroups}) and dimer groups   (Section~\ref{newDimer coverings+}).   In Section~\ref{extension-----} we discuss connections with braid groups. In Section~\ref{extension}  we extend our definitions and  results   to hypergraphs.

The author would like to thank  M. Ciucu for several stimulating discussions. This work   was partially supported by the NSF
  grant  DMS-1202335.

 \section{Glides}\label{sect-prelim}

 \subsection{Gliding systems} \label{BU} By a  {\it  gliding system} in a group $G$ we mean a pair of sets
 $(\G\subset G \setminus \{1\} , \I\subset \G\times \G)$ satisfying the following conditions:

 (1) if $s\in \G$, then  $s^{-1}\in \G$;

 (2) if $(s , {t}) \in \I$ with $s,{t}\in \G$, then $s {t}={t}s$ and $(s^{-1} , {t})\in
 \I $, $({t},s)\in
 \I$;

 (3) $(s,s)\notin \I$ for all $s\in \G$.

  The elements of $\G$ are called
  {\it glides}.   The inverse of a glide is a glide while the unit   $1\in G$ is never a glide. We say that two glides $s,{t} $ are {\it
 independent} if  $(s , {t}) \in \I$. By (2), the independence is a
property of   non-ordered pairs of glides preserved under inversion
of one or both glides. A glide is never independent from itself or
its inverse. Also, independent glides commute.

For   $s\in \G$  and    $A \in G$,  we say that $sA\in G$ is
obtained from $A$ by {\it (left) gliding along}~$s$. One can
similarly consider right glidings but we do not need them. Clearly,
  the gliding of $sA$  along $s^{-1}$ yields $s^{-1}sA=A$.

A {\it set of independent glides}  is  a set   $S \subset \G$ such
that $(s,t)\in \I$ for any distinct $s,t\in S$.  Since independent
glides commute, a  finite set of independent glides $S$ determines
an element $[S]=\prod_{s\in S} s$ of $ G$. By definition,
$[\emptyset]=1$.

\subsection{Examples}\label{exam} 1.   For any group $G$, the following pair is a   gliding
system:  $\G= G \setminus \{1\}$  and $\I$ is the set of all pairs
$(s,t)\in \G\times \G$ such that $s\neq t^{\pm 1}$ and $st=ts$.

2. The pair $(\G= G \setminus \{1\}, \I=\emptyset)$ is a   gliding
system in a  group~$G$.

3. Let $G $ be a free abelian group   with   free commuting generators $\{g_i\}_i$. Then
$$\G=\{ g_i^{\pm 1} \}_{i}, \quad \I=\{(g_i^\varepsilon,g_j^\mu)
\, \vert \, \varepsilon=\pm 1,
\mu =\pm 1,   i\neq j\}$$ is a
  gliding system in $G$.

4. A   generalization of the previous example is provided by the
theory of right-angled Artin groups (see \cite{Ch} for an
exposition). A right-angled Artin group is a group allowing a
presentation  by generators and relations in which all relators are
commutators of the generators. Any graph ${\Gamma}$ with the set of
vertices $V$ determines a right-angled Artin group $G=G({\Gamma})$
with generators $\{g_s\}_{s\in V}$   and relations $g_s g_t g_s^{-1}
g_t^{-1}=1$ which hold whenever   $s, t \in V$ are connected by an
edge in~${\Gamma}$ (we write then $s\leftrightarrow t$). Abelianizing
$G$ we obtain that $g_s\neq g_t^{\pm 1}$ for $s\neq t$. The pair
\begin{equation}\label{CanA} \G=\{ g_s^{\pm 1} \}_{s\in V}, \quad
\I=\{(g_s^\varepsilon, g_t^\mu)\, \vert \, \varepsilon=\pm 1, \mu
=\pm 1, s,t \in V, s\neq t, s\leftrightarrow t
 \}\end{equation} is a
gliding system in $G$.

5.  Let $E$ be  a set  and     $ G= 2^E$   be the power set of $E$ consisting of all
subsets of $E$. We define multiplication  in $G$ by  $ A  B = (A\cup
B)\setminus (A\cap B)$  for   $A,B\subset E$. This   turns $G$
into an abelian   group with unit $1=\emptyset$, the {\it power group of
  $E$}.  Clearly, $A^{-1}=A$ for all $A\in G$. Pick any set
$\G\subset G\setminus \{1\}$ and declare
  elements of~$\G$   independent when  they are disjoint as
subsets of~$E$. This gives  a     gliding system in~$G$.

6.  Let $E$ be a set, $H$ be a multiplicative group, and  $G=H^E$ be the group of
all maps $E\to H$ with pointwise multiplication. The {\it support}
of  a map $f:E\to H$ is the
 set  $ f^{-1} (H\setminus \{1\})\subset E$. Pick a   set $\G\subset G\setminus \{1\}$ invariant under inversion and  declare
  elements of $\G$   independent if
   their supports are disjoint. This gives
  a
gliding system in $G$. When $H$ is a cyclic group of order~$2$, we
recover Example 5 via the isomorphism  $H^E\cong 2^E$ carrying a map
$E\to H$ to its support.


\section{ Glide    complexes and glide groups}\label{ From glides to cubed   complexes}

We discuss   cubed complexes associated with gliding systems. We
begin with generalities on cubed complexes referring for details to
\cite{BH}, Chapters I.7 and II.5.

\subsection{Cubed complexes}\label{Cubed complexes} Set $I=[0,1]$. A {\it cubed complex}  is a
CW-complex $X$   such that each (closed) $k$-cell of $X$ with $k\geq
0$ is a continuous map   from the $k$-dimensional cube $I^k$ to $X$
whose restriction  to the interior of $I^k$ is injective and whose
restriction  to each $(k-1)$-face of $I^k$ is an isometry of that
face onto $I^{k-1}$ composed with a $(k-1)$-cell $I^{k-1}\to X$ of
$X$.  The $k$-cells $I^k \to X$   are not required to be injective.
The $k$-skeleton  $X^k$  of $X$ is the union of the images of all
 cells of dimension  $ \leq k$.

 For example, the cube $I^k$ together with all its faces is a cubed complex.
 The $k$-dimensional torus obtained by identifying opposite faces of $I^k$ is a cubed complex.

The {\it link} $LK(A)=LK(A;X)$ of a 0-cell $A $ of a cubed complex
$X$ is the space of all directions at $A$. Each triple ($k\geq 1$, a
vertex $a$ of $I^k$,   a $k$-cell $\alpha:I^k \to X $ of $X$
carrying $a$ to $A$) determines a  $(k-1)$-dimensional simplex in
$LK(A)$ in the obvious way. The faces of this simplex are determined
by the restrictions of $\alpha$ to the faces of $I^k$
containing~$a$. The simplices corresponding to all triples
$(k,a,\alpha)$ cover $LK(A)$ but may not form a simplicial complex.
We say, following \cite{HW}, that the cubed complex $X$ is {\it
simple} if the links of all $A\in X^0$ are    simplicial complexes,
i.e., all   simplices  in $LK(A)$    are embedded and
the intersection of any two simplices is a common face.


A {\it flag complex} is a simplicial complex such that any finite
collection of pairwise adjacent vertices spans a simplex.  A cubed
complex is {\it nonpositively curved} if it is simple and the link
of each 0-cell is a flag complex. A theorem of M. Gromov asserts
that the universal covering of any connected finite-dimensional
nonpositively curved cubed complex  $X$ is   a CAT(0)-space.
Since
CAT(0)-spaces are contractible, all higher homotopy groups of  such an $X$
vanish while the fundamental group $\pi=\pi_1(X)$ is torsion-free. This group
     satisfies a strong form of the Tits alternative: each
subgroup of $\pi $ contains a rank 2 free subgroup or is virtually a
finitely generated abelian group, see \cite{SW}. Also,  $\pi$ does
not have Kazhdan's property (T), see \cite{NR1}. If $X$ is  compact,
then $\pi$ has solvable word and conjugacy problems and is
biautomatic, see \cite{NR2}.


\subsection{Glide  complexes}\label{cubecomplexX} Consider   a
group $G$ endowed with a gliding system. We associate with $G$ a cubed complex $
X_G$ called the {\it glide complex}.

We call a  set of glides $S$   {\it cubic} if $S$ is finite, the  glides in
$S$ are pairwise independent,  and for any distinct subsets
$T_1,T_2$ of $  S$ we have  $[T_1]\neq [T_2]$.    The following
properties of cubic sets of glides are straightforward:

-- any  set  of independent glides with $\leq 2$ elements is
cubic;

-- all subsets of a cubic set of glides are cubic;

-- for each subset $T$ of a cubic set of glides $S$, the set
  $S_T=(S\setminus T)\cup \{t^{-1}\}_{t\in T}$ is a  cubic set of glides.

A {\it based cube} in $G$ is a pair  ($A\in G$, a cubic set of
 glides $S\subset G$). Then, the set $\{[T]A\}_{T\subset S} \subset G$ has
 $2^k$ distinct
elements where $k=\card(S)$ is the {\it dimension} of the based cube
$(A,S)$.  These $2^k$ elements of $G$ are called  the {\it vertices}
of  $(A,S)$.

  Two based cubes $(A,S)$ and $(A', S')$ are   {\it equivalent} if there is a set $T\subset S$ such that
$A'=[T]A$ and $S'=S_T$ where $S_T$ is   defined above. This is
indeed an equivalence relation on the set of based cubes. It is
clear that   each $k$-dimensional based cube is equivalent to $2^k$
based cubes (including itself). The equivalence classes of
$k$-dimensional based cubes are called {\it $k$-dimensional  cubes} or {\it $k$-cubes}
in $G$. Since equivalent based cubes have the same vertices, we may
speak of the vertices of a cube. The $0$-cubes in $G$ are just the elements of $G$.

A cube $Q$ in $G$ is a {\it face} of a cube $Q'$  in $G$  if $Q,Q'$ may be
represented by based cubes $(A,S)$, $(A', S')$, respectively, such
that $A=A'$ and $S\subset S'$. Note that  in the role of $A$ one
may take an arbitrary vertex of $Q$.

The  glide complex  $X_G$ is the cubed complex obtained by taking a
copy of $I^{k}$ for each    $k$-cube    in $G$ with
$k\geq 0$ and gluing these copies via
 identifications determined by inclusions of     cubes   into bigger     cubes   as their faces.
Here is a more precise definition.   A point of $X_G$ is represented
by a triple $(A,S,x)$ where $(A,S)$ is a  based cube in $G$ and  $x$
is  a map $ S \to I $. For fixed $(A, S)$,    such  maps $x$ form a
geometric cube, a product of $\card (S)$ copies of $I$.  We take a
disjoint union of these cubes over all   $(A,S )$ and factorize it
by the equivalence relation generated by the following relation:
$(A,S,x) \sim (A',S',x')$ when $ A=A', S\subset S', x=x'\vert_{S},
x'(S'\setminus S)=0$ or there is a set $T\subset S$ such that
$A'=[T]A, S'=S_T $, $x'=x$ on $S\setminus T$, and
$x'(t^{-1})=1-x(t)$ for all $ {t\in T}$.  The  quotient space $X_G$
is a cubed space in the obvious way. It is easy to see that all cubes forming $X_G$ are embedded in $X_G$.

The formula $(A,S,x)g=(Ag,S,x)$ for   $g\in G$ defines a right
action of $G$ on~$X_G$ preserving the cubed structure   and free on
the $0$-skeleton $X^0_G=G$. When $G$ has no elements of finite  order, the action of $G$ on $X_G$ is free. This follows
  from the fact that the set of vertices of the minimal cube
  containing a fixed point of the action of $g\in G$ must be invariant under
  the action of $g$.

We further associate with any
 set $\E\subset G$    the  cubed complex $X_\E \subset X_G$ formed
by the cubes in $G$ whose   vertices lie in $\E$.  Such cubes
are said to be {\it cubes in $\E$}.   The   set of   connected
components of $ X_\E$  is the quotient of $\E$ by the equivalence
relation generated by   glidings in $\E$. We call $X_\E$  the {\it
glide complex} of $\E$. The fundamental groups of the components of
$X_\E$ are called the {\it glide groups  of} $\E$. Thus, each $A\in
\E$ gives rise to a glide group $\pi_1(X_\E, A)$, and  elements of
$\E$ related by glidings  in $\E$  give rise to isomorphic groups.
 If $\E$ is finite, then $X_\E$ is a finite
CW-space and the glide groups are finitely presented.

 \begin{lemma}\label{simple}   The  cubed complex $X_\E$ is simple for any $\E\subset G$.
   \end{lemma}

\begin{proof} A neighborhood of $A\in G$ in $X_G$ can be obtained by taking all triples $(A,S,x)$, where $S$ is a  cubic set of glides and $x(S)\subset [0,1/2)$, and  identifying two such triples $(A,S_1,x_1)$, $ (A,S_2,x_2)$ whenever $  x_1=x_2$ on $S_1\cap S_2$ and $
x_1(S_1\setminus S_2)=x_2(S_2\setminus S_1)=0$. Therefore, the link $LK_G(A)$ of $A $ in $X_G$ is the simplicial complex whose vertices are  glides   and whose simplices are   cubic sets of  glides. The link $LK_\E(A)$ of $A $ in $X_\E$ is the subcomplex of $LK_G(A)$ formed by the glides $s\in G$ such that $sA\in \E$ and the cubic sets $S$ such that $[T]A\in \E$ for all $T\subset S$.
\end{proof}

For $A\in \E $, we can reformulate the flag condition on
$LK_\E(A)$ in terms of glides. Observe that the  sets of pairwise
 adjacent vertices in $LK_\E(A)$ bijectively correspond to
 the  sets of independent glides $S \subset G$  satisfying the
following condition:

$(\ast)$  $sA\in \E$ for all $s\in S$  and $s{t} A\in \E$ for all
distinct $s, {t}\in S$.

\begin{lemma}\label{simple+}   $LK_\E(A)$ is a flag complex if and only if any   finite set
 of independent glides $S \subset G$  satisfying $(\ast)$ is cubic and   $[S]A\in \E$.
   \end{lemma}

This lemma follows directly from the definitions. One should use the
obvious fact  that if a  set  of independent glides satisfies
$(\ast)$ then so do all its subsets.

 We call a set $\E\subset G$   {\it regular}
if    for every $A\in \E$, all   finite sets
 of independent glides  $S \subset G$  satisfying $(\ast)$ are cubic.



 \begin{theor}\label{simpleNEW}  For a  group $G$ with a gliding system
 and   a  set $\E\subset G$,
 the  cubed complex $X_\E$ is nonpositively curved if and only if $\E$   is regular and meets the following  3-cube condition:  if $A\in \E$ and    pairwise
independent glides $s_1, s_2, s_3\in G$ satisfy $  s_1A, s_2 A, s_3
A, s_1 s_2 A, s_1 s_3 A, s_2 s_3A\in \E$, then $s_1 s_2 s_3 A\in
\E$.
\end{theor}

The
  3-cube condition on $\E$  may be reformulated by saying that if seven vertices of a 3-cube in $G$ belong to $\E$, then so does the  eighth vertex. This condition is always fulfilled when $\E=G$ or, more
generally, when $\E$ is a subgroup of $G$.


 \begin{proof}  Lemmas \ref{simple}  and \ref{simple+} imply that  $X_\E$ is
nonpositively curved if and only if for every $A\in \E$,
 each    finite set   of independent
glides $S$    satisfying $(\ast)$ is cubic and    $[S]A\in \E$. We
must only show that the inclusion $[S]A\in \E$ can be replaced with
the   3-cube condition.
 One direction is obvious: if $A, s_1, s_2, s_3\in G$ satisfy the assumptions of  the 3-cube condition, then
the set  $S=\{s_1,s_2, s_3\}$  satisfies $(\ast)$ and so
$s_1s_2s_3 A=[S]A\in
  \E$.
 Conversely, suppose that $\E$ is regular and meets the 3-cube condition.  We must  show
that $[S]A\in \E$ for any $A\in \E$ and any   finite set
 of independent glides $S \subset G$  satisfying $(\ast)$.   We proceed by induction on
$k=\card(S)$. For $k=0$, the claim follows from the inclusion $A\in
\E$. For $k=1,2$, the claim follows from $(\ast)$. If $k\geq 3$,
then the induction assumption guarantees that $[T]A\in \E$ for any
proper subset $T\subset S$. If $S=\{s_1,...,s_k\}$, then applying
the 3-cube condition   to $s_1,s_2,s_3$ and the element $s_4 \cdots
s_{k}A $ of $\E$, we obtain that $[S]A\in \E$.
 \end{proof}

 \subsection{Regular gliding systems}\label{regglsy}  A gliding system in a group $G$ is {\it regular} if all finite sets of independent glides in $G$  are cubic. This condition implies that
 all
subsets of $G$ are regular.  Though we   treat arbitrary gliding systems, we are mainly interested in regular ones.
Without much loss, the reader may focus on the regular   gliding systems.

  \begin{corol}\label{simpleNEWBIS}  For a  group $G$ with a regular gliding system,
the glide complex of    a  set $\E\subset G$  is nonpositively curved if and only if $\E$ satisfies the 3-cube condition.
\end{corol}

By Section \ref{Cubed complexes}, if $X_\E$ is nonpositively curved,
then the universal covering   of every finite-dimensional component
of $X_\E$ is a CAT(0)-space. The component itself is then an
Eilenberg-MacLane space of type $K(\pi,1)$  where $\pi$ is the
corresponding glide group. Note that  $\dim X_\E$ is the maximal dimension of a cube in~$\E$.

 \begin{corol}\label{simpleNEWBIS++++}  For a  group $G$ with a regular gliding system,
the glide complex $X_G$  is nonpositively curved.
\end{corol}

 \subsection{Examples}\label{examsec2}
 It is easy to see that the
 gliding systems  in Examples  \ref{exam}.2--6 are  regular.    Corollaries \ref{simpleNEWBIS} and \ref{simpleNEWBIS++++} fully
 apply to these   systems.  The gliding system   in Example   \ref{exam}.1, generally speaking, is not regular.

   Here are a few further remarks on these examples.

  In (\ref{exam}.2),  $X_G $ is the complete
  graph with the set of vertices $G$.  The    action of $G$ on $X_G$
 is free
 if and only if $G$ has no elements  of order 2.  All subsets of $G$ satisfy the
 3-cube condition and the corresponding glide groups are free.

  In (\ref{exam}.3), if  the rank, $n$, of $G$ is finite, then  $X_G=\RR^n$.


 In (\ref{exam}.4), $X_G$ is simply-connected,   the   action of $G$ on $X_G$
 is free, and  the projection $ X_G\to X_G/G$ is the universal covering
of
 $X_{G}/G$. Composing the cells  of $ X_G$ with this projection we turn $X_G/G$ into a cubed complex  called the
{\it Salvetti complex}, see \cite{Ch}.  This complex has only one
0-cell whose link   is isomorphic to the link of any vertex
of $X_G$ and  is a flag complex. This recovers the well known fact
that the Salvetti complex is nonpositively curved. Clearly, $\dim X_{G}
=\dim (X_{G}/G)$ is the maximal
  number of  vertices of a  complete subgraph   of~${\Gamma}$.

In (\ref{exam}.5),  all elements of $G$ have order 2 and the     action of $G$ on $X_G$
 is not free.  If the set $E$ is finite in  (\ref{exam}.5),  (\ref{exam}.6), then $X_G$ is finite dimensional.

In Example  \ref{exam}.3  with $n\geq 3$ and  in Examples \ref{exam}.4--6, the 3-cube condition
  holds for some but not all~$\E \subset G $.

 \section{Inclusion homomorphisms}\label{incmaps}

In   the setting of Section~\ref{cubecomplexX}, a subset    $ \F$ of $
\E $ determines a subcomplex
   $ X_{\F}$ of~$ X_{\E}$. We formulate    conditions   ensuring that the   inclusion $X_{\F}\hookrightarrow X_{\E}$  induces an injection of the
fundamental  groups.

   \subsection{The square condition}\label{incl++} Let $G$ be a group   with a
   gliding system and  $\E\subset G$. We say that a  set $\F  \subset \E$
   satisfies the {\it square condition   $\rel \E$}, if for any $A\in \F$ and
   any independent glides $s,t\in G$ such that $sA, tA\in \F, stA \in
   \E$, we necessarily have $stA\in \F$. For example, the
   intersection of $\E$ with any subgroup of~$G$  satisfies the  square condition   $\rel
   \E$.

 \begin{theor}\label{ecu} Let $\E\subset G$ be a regular set    satisfying  the 3-cube condition  and such that $\dim X_\E<\infty$. Let $\F$ be a subset of $\E$
 satisfying the square condition $\rel \E$. Then  the
 inclusion homomorphism $\pi_1(X_{\F}, A) \to \pi_1(X_\E, A)$ is
 injective for all $A\in \F$.
\end{theor}

Theorem~\ref{ecu}  is proven in Section \ref{ecu+} using  material
of Section \ref{llocal}.

We say that a set $\E\subset G$ satisfies the {\it square condition} if it satisfies the square
condition $\rel G$ that is   for any $A\in \E$ and
   any independent glides $s,t\in G$ such that $sA, tA\in \E $, we necessarily have $stA\in \E$.

\begin{corol}\label{simpleNEW++-}  If  $\dim X_G<\infty$, then for every   regular set $\E\subset G$ satisfying the square
condition  and every   $A \in  \E$,   the
 inclusion homomorphism $\pi_1(X_{\E}, A) \to \pi_1(X_G, A)$ is
 injective.
\end{corol}

It would be interesting to find out whether the conditions $\dim
X_\E<\infty$ and $\dim X_G<\infty$ in these statements are
necessary.

\subsection{Cubical maps}\label{llocal}  Let $X$ and $Y$ be simple cubed complexes. A {\it cubical map}
$f: X\to Y$
  is a continuous map whose composition
with any $k$-cell $I^k\to X$ of $X$ splits as   a composition of a
self-isometry of $I^k$ with a $k$-cell $I^k\to Y$ of $Y$ for all
$k\geq 0$. For every  $A\in X^0$, such a   map $f $ induces a
simplicial
  map $f_A:LK(A)\to LK(f(A))$.   The cubical map $f $ is a {\it local isometry} if  for all $A\in X^0$, the map $f_A$
is an embedding onto a full subcomplex of $LK(f(A))$. Recall that a
simplicial subcomplex $Z'$ of a simplicial  complex $Z$ is  {\it  full}  if any simplex of
$Z$ with vertices in $Z'$ entirely lies in $Z'$.

 \begin{lemma}\label{Wise} \cite[Theorem 1.2]{CW} If $f:X\to Y$ is a local isometry of nonpositively curved finite-dimensional
 cubed
 complexes, then the induced homomorphism $f_*: \pi_1(X,A)\to \pi_1(Y,
 f(A))$ is injective for all $A\in X$.
   \end{lemma}

   \subsection{Proof of Theorem~\ref{ecu}}\label{ecu+}   Since $\E$ is regular,
   so is $\F\subset \E$. The 3-cube condition on $\E$ and the square condition on
   $\F$     imply that $\F$ satisifes the 3-cube
   condition. By Theorem~\ref{simpleNEW},   $X_\E$ and $X_{\F}$ are nonpositively
   curved. By
   assumption, the cubed complex $X_\E$ is finite-dimensional and so
   is its subcomplex $X_{\F}$. We claim that the inclusion $X_{\F}\hookrightarrow
   X_{\E}$ is a local isometry. Together with Lemma~\ref{Wise} this
   will imply the   theorem.

To prove our claim,   pick any $A\in \F$ and consider the simplicial
complexes $L'=LK(A, X_{\F})$, $L=LK(A, X_{\E})$.  The inclusion
$X_{\F}\hookrightarrow
   X_{\E}$ induces an embedding $L'\hookrightarrow L$, and we need only to verify that the image of $L'$ is  a full
subcomplex of~$L$.  Since   $L'$ is a  flag simplicial
complex, it suffices to verify that any  vertices of~$L'$
adjacent in~$L $ are  adjacent in~$L'$.
This follows
   from the square condition  on~$\F$.

   \section{Typing homomorphisms}\label{embArtin}

\subsection{Artin
groups  of glides}\label{prelimArtin++}   A group $G$ carrying  a
gliding system $(\G,\I)$ determines a right-angled Artin group
$\A=\A(G)$ with generators $\{g_s\}_{s \in \G }$ and    relations $g_s
g_t g_s^{-1} g_t^{-1}=1$ when $(s, t)\in \I $. This group   is
associated with the graph whose vertices are the glides in $G$ and
whose edges connect independent glides. By Example \ref{exam}.4, the group $\A $  carries a
gliding system with glides    $\{g_s^{\pm 1}\}_{s\in \G}$. Two glides
$g_s^{\pm 1}$ and $g_t^{\pm 1}$ in $\A$  are independent if
and only if $s\neq t$ and $(s,t)\in \I$.

We shall relate the group $\A $ to the glide groups associated with a  set $\E
\subset G$. To this end,  we introduce a notion of an orientation on
  $\E $. An {\it orientation} on   $\E$ is  a
choice of direction on each 1-cell of the CW-complex $X_\E$ such
that    the opposite sides of any (square) 2-cell of $X_\E$ point
towards each other on the boundary loop. This means that  for any
based square  $(A,\{s,t\})$ with $A, sA, tA, stA\in \E$, the 1-cells
connecting $A$ to $sA$ and $tA$ to $stA$ are directed either towards
$sA, stA$ or towards $A,tA$ (and similarly with $s$ and $t$
exchanged). A set $\E $ is {\it orientable} if it has at least one
orientation and is {\it oriented} if it has a
distinguished orientation. An orientation of $\E$ induces an
orientation of any  subset  $\F\subset  \E$ via the inclusion
$X_{\F}\subset X_\E$. Therefore, all subsets of an orientable set
are orientable.  These definitions   apply, in particular,  to $\E=G$. Examples of oriented sets will be given in Section~\ref{lloccgccgcgl}.


  Let $\E\subset G$ be an oriented set. We assign to each 1-cell $e$ of $X_\E$ the glide
$\vert e\vert= BA^{-1}$ where $A\in \E$ is the initial endpoint of
$e$ and $B\in \E$ is the terminal endpoint of $e$ with respect to
the distinguished orientation. Consider a  path $\alpha$
in $X_\E$ joining two 0-cells and  formed by $n\geq 0$ consecutive
 1-cells $e_1,...,e_n$ oriented so that the terminal endpoint of
$e_k$ is the initial endpoint of $e_{k+1}$ for $k=1, ..., n-1$. The
orientation of   $e_k$   may coincide or not with that
given by the orientation of $\E$. We set $\nu_k=+1$ or $\nu_k=-1$,
respectively. Set
\begin{equation}\label{pathee} \mu(\alpha)=  g_{\vert e_1\vert}^{\nu_1} \,  g_{\vert e_2\vert}^{\nu_2} \cdots g_{\vert e_n\vert}^{\nu_n} \in \A=\A(G)  .\end{equation}
 It is clear that   $\mu(\alpha)$ is preserved under  inserting in
the sequence $e_1,...,e_n$ two opposite 1-cells  or four
1-cells forming the boundary of a 2-cell. Therefore   $\mu(\alpha)$ is
preserved under homotopies of $\alpha$ in $X_\E$ relative to the
endpoints. Applying $\mu$ to loops
based at $A\in \E$, we obtain a   homomorphism $\mu_A: \pi_1(X_\E,
A)\to \A $. Following the terminology of \cite{HW}, we call $\mu_A$ the {\it typing homomorphism}.

\begin{theor}\label{simpleNEW++}  If there is  an upper bound
on the number
 of pairwise independent glides in $G$, then for any   oriented regular    set $\E\subset G$  satisfying  the square condition and any $A\in \E$, the typing homomorphism  $\mu_A: \pi_1(X_\E,
A)\to \A $  is an injection.
\end{theor}



\begin{proof} 
 Consider the
cubed complex $X=X_{\A}$  associated with the gliding system in~$\A$ and the Salvetti complex $Y=X /\A$.  Recall that $\A=\pi_1(Y, \ast)$ where $\ast$ is the unique 0-cell of $Y$.     As explained in Section \ref{examsec2}, both
$X $ and $Y$ are nonpositively curved. The assumptions of the
theorem imply that the cubed complex $  X_\E $ is nonpositively
curved, and the spaces $X $, $Y$, $ X_\E $   are finite-dimensional. We claim
that the homomorphism $   \mu_A: \pi_1(X_\E, A)\to
\A =\pi_1(Y,\ast)$ is induced by a local isometry $X_\E\to Y$.
Together with Lemma~\ref{Wise}  this will imply the theorem.

Since the gliding system in $\A$ is regular, the points of $X $ are represented by  triples $(A\in
\A, \sal , x:\sal \to I)$ where $\sal$ is a  finite set of independent
glides in $\A$  that is     $S=\{g_s^{\varepsilon_s}\}_s$ where
$s$ runs over a finite set of   independent glides in $G$ and
$\varepsilon_s\in \{\pm 1\}$. The space $X $ is obtained by   factorizing  the set of
  such triples
  by the equivalence relation defined in Section~\ref{cubecomplexX}. The space $Y$ is obtained from $X $ by forgetting the first term, $A$, of the triple.
A point of $Y $ is represented by a  pair   (a  finite set of independent
glides $\sal\subset \A$, a map  $x:\sal \to I$). The space $Y $  is obtained by   factorizing  the set of
  such pairs by the equivalence relation generated by the
following relation: $( \sal,x) \sim ( \sal',x')$ when $
\sal\subset \sal', x=x'\vert_{\sal}, x'(\sal'\setminus \sal)=0$ or
there is   $T\subset \sal$ such that $  \sal'=\sal_T $,
$x'=x$ on $\sal\setminus T$, and $x'(t^{-1})=1-x(t)$ for all $ {t\in
T}$.

We now construct a cubical map $f:X_\E\to Y$.   The idea  is to map
each 1-cell, $e$,    of  $X_\E$   onto the 1-cell of $Y$ determined by
$g_{\vert e\vert} $ where $\vert e\vert \in G$ is the glide determined by the distinguished orientation of $e$  as
above. Here is a detailed definition of $f$. A point $a\in X_\E$ is
represented by a triple $(A\in \E, S ,x:S\to I)$ where
$S $ is a cubic set of glides in $G$ such that $[T]A\in \E$ for
all $T\subset S$. For  $s\in S$, set $\vert s \vert =\vert e_s
\vert$ where $e_s$ is the 1-cell of $X_\E$ connecting $A$ and $sA$.
By definition, $\vert s \vert =s^{\varepsilon_s}$ where
$\varepsilon_s=+1$ if $e_s$  is oriented towards $ sA$ and
$\varepsilon_s=-1$ otherwise. Let $f(a)\in Y$ be the point
represented by the pair  $(  \sal=\{g_{\vert s
\vert}^{\varepsilon_s}\}_{s\in S}, y:\sal \to I)$ where $y(g_{\vert
s \vert}^{\varepsilon_s})=x(s)$ for all $s\in S$. This yields a well-defined cubical map $f:X_\E\to Y$ inducing
$\mu_A$ in $\pi_1$.

It remains to show that $f$ is a local isometry.  The link  $L_A=LK_\E(A)$  of   $A \in \E$ in $X_\E$ was described  as a simplicial complex  in the proof of Lemma~\ref{simple}. It has
a vertex $v_s$ for every glide $s\in G$ such that $sA\in \E$. A set of
vertices $\{v_{s}\}_s$ spans a simplex in $L_A$
 whenever   $s$ runs over a cubic set of glides.
The link, $L$,  of
$\ast  $ in $Y$  has
two vertices $w^+_s$ and $w^-_s$ for every glide $s \in G $. A set of
vertices $\{w^{\pm}_{s}\}_s$ spans a simplex in $L$
 whenever   $s$ runs over a finite set of independent glides.   The map $f_A: L_A\to L$ induced by $f$ carries  $v_s$ to $w_{\vert s
\vert}^{\varepsilon_s}$. This map   is an embedding since we can recover   $s $ from
${\vert s
\vert}$ and ${\varepsilon_s}$. The square condition  on~$\E$ implies that any  vertices of~$L_A$
adjacent in~$L $ are  adjacent in~$L_A$.
Since   $L_A$ is a  flag simplicial
complex,   $f_A(L_A)$ is a full subcomplex of $L$.
\end{proof}

\begin{corol}\label{simpleNEW++-eee}  If  the gliding system in $G$ is regular, the number
 of pairwise independent glides in $G$ is bounded from above,  and $G$ is orientable in the sense of Section~\ref{prelimArtin++}, then        $\mu_A:\pi_1(X_G, A)\to \A $ is an injection for all $A\in G$.
\end{corol}

This follows from Theorem \ref{simpleNEW++} because  the square condition on $G$ is void.

\subsection{Applications}\label{Applications}  Finitely generated right-angled Artin groups  are biorderable, residually
nilpotent, and   embed in $SL_n(\ZZ)$ for some $n$ (so are residually
finite), see \cite{DT}, \cite{DJ}, \cite{CW},  \cite{HsW}. These properties are
hereditary and   are shared by all subgroups of finitely
generated right-angled Artin groups. 
Combining with Theorem~\ref{simpleNEW++} we obtain the following.

\begin{corol}\label{simpleNEW++-nilp}  If the set of glides in $G$ is finite and an orientable   regular set $\E\subset G$  satisfies  the square condition,
 then for all $A\in \E$,  the group   $\pi_1(X_\E, A)$ is  biorderable, residually nilpotent, and
 embeds in $SL_n(\ZZ)$ for some $n$.
\end{corol}

\subsection{Remark} It would be interesting to deduce  Theorem \ref{simpleNEW++}   from the results of \cite{HW}. Are the glide complexes A-special in the sense of \cite{HW}?

\section{Presentations of glide groups}\label{redu}

Throughout this section, we fix   a group $G$ carrying a gliding
system and a set $\E\subset G$.   Under certain assumptions, we
obtain a presentation  of the glide group of~$\E$ by generators and
relations.

\subsection{Hulls}\label{mincu} Given  a set $J\subset \E$ and a
cube $Q $ in $ \E$ whose set of vertices contains $J$, we call $Q$ the   {\it hull of $J$ in
$\E$} if $Q$ is a   face of all cubes in $\E$ whose sets of vertices contain  $J$. We refer to
Section \ref{cubecomplexX} for the   definitions of cubes in $ \E$, faces, etc. If
  a hull of $J$
  exists,
    then it is   unique. If $J $ has only one element, then $J$ is a 0-cube and is its own hull.


 \begin{theor}\label{phi-+}    If    each
 2-element subset of $  \E$ has a hull  in $ \E$, then $X_\E$ is connected and for each $A_0\in \E
  $,
   the group $\pi_1(X_\E, A_0)$ is canonically isomorphic to  the group with generators $\{y_{A,B}\}_{A,B\in \E}$ subject to the following relations:
  $y_{A,C}=y_{A,B}\,  y_{B,C}$ for every triple   $A,B,C $ of vertices of a  cube in $\E$ and  $y_{A_0,A}=1$ for all $A\in \E$.
  The image of each $
y_{A,B}$ in $\pi_1(X_\E, A_0)$
  is represented by the loop in $X_\E$
composed of   a path from $A_0$ to $A$ in the hull of $\{A_0,A\}$, a
path from $A $ to $B$ in the hull of $\{A,B\}$, and a path from $B$
to $A_0$ in the hull of $\{A_0,B\}$. \end{theor}

Theorem~\ref{phi-+} is proved in Section  \ref{pro-1} using Sections
\ref{grF} and  \ref{pro-2}  which are concerned with   more general
situations. 
One can check that  $A,B,C\in
\E$ are vertices of a certain cube in $\E$ if and only if there is $K\in \E$ and
a cubic set of   glides $S\subset G$ such that $[T]K\in \E$ for all
$T\subset S$ and $S$ has a partition $S=X \amalg Y \amalg Z$ such
that $A=[X]K$, $B=[Y] K$, and $C=[Z]K$.


\subsection{The group $F$}\label{grF} Let $F$ be the
group generated by symbols labeling   1-cells of $X_\E$ subject to
the relations associated with   2-cells of $X_\E$. More precisely,
$F$ is generated by  the set
$$\{x_{A,B}\,\vert \, A,B\in \E \,\, {\rm {such \,\, that}} \,\,  AB^{-1}\,\,  {\rm {is\,\,  a \,\,  glide}} \}.$$
subject to the   relations $x_{A,B} \, x_{B,A}=1$ for any $A,B$ and
\begin{equation}\label{pathpp}x_{A,B} \,  x_{B,C}\,  x_{C,D} \, x_{D,A} =1\end{equation}  for any
$A,B,C,D\in \E$   such that $B= sA,C=s tA,
D=tA $ for some independent glides $s,t\in G$. For a  path $\alpha$ in $X_\E$ formed by 1-cells and consecutively
connecting   0-cells $A_0$, $A_1$,..., $A_n  \in \E$, set
\begin{equation}\label{path} \phi(\alpha)= \prod_{k=1}^{n} x_{A_{k-1}, A_{k}} \in F.\end{equation}
It is clear that $\phi(\alpha)$ is preserved under homotopies of
$\alpha$ relative to the endpoints.


 \begin{lemma}\label{phi--}  Applying $\phi$ to loops based at $A_0\in \E$, we
obtain  a      homomorphism $ \phi:\pi_1(X_\E, A_0)\to F$.  This  homomorphism   is a split injection.
   \end{lemma}

\begin{proof}   The first claim is obvious.   To prove the second claim, denote by $X$  the component of $X_\E$ containing  $A_0$.  For each  $0$-cell
 $A\in X$   pick a path $\rho_A$   from $A_0 $ to $A $ in $X$. For $\rho_{A_0}$ we take the constant path at $A_0$.
 Given $0$-cells $A,B\in X$ such that $AB^{-1}$ is a glide, denote by $\overline {AB}$ the path from $A$ to $B$
  running along the  unique 1-cell of $X $
connecting   $A$ and $B$.   Set $$\psi(x_{A,B})=[\rho_A \overline
{AB} \rho_B^{-1}]\in \pi_1(X, A_0)=  \pi_1(X_\E, A_0)$$ where the
square brackets stand for the homotopy class of a loop. For
$0$-cells $A,B$ not lying in $X$, set $\psi(x_{A,B})=1$.  The formula
$x_{A,B}\mapsto \psi(x_{A,B})$ is compatible with the defining
relations of   $F$  and yields  a homomorphism $ \psi:F\to
\pi_1(X_\E, A_0)$.  It follows from the definitions that $ \psi\phi =\id$.
\end{proof}

Though we shall not need it, note that for oriented   $\E$, there is
a homomorphism $g: F\to \A(G)$ carrying $ x_{A,B}$ and $x_{B,A}$
respectively to $g_{AB^{-1}}$ and $ (g_{AB^{-1}})^{-1}$ when there is a 1-cell of $X_\E$ oriented from $A$ to $B$. Then  $ g\phi=
\mu_{A_0}:\pi_1(X_\E, A_0)\to \A(G)$.

\subsection{Central  points}\label{pro-2}
  We call   $A_0 \in \E
 $ {\it central}  if  for any
  $A\in \E$,  the set  $\{A_0,A \}$ has a hull  in $\E$. This means that    there is a  cubic set of  glides $S\subset G$    such that  $A\in  \{[T]A_0\}_{T\subset S}\subset \E
 $  and that any   cubic set of glides with these properties contains~$S$. The vertices $A$,  $A_0$  of the based cube $(A_0,S)$ form then a big diagonal and so      $A=[S]A_0$.
    The  vertices   $\{s^{-1}A\}_{s\in S}$ of this cube are called the {\it ascendents} of $A$. The set of   ascendents of $A $ is empty if and only if $A=A_0$.




  \begin{lemma}\label{phi-+c}    If $A_0\in \E $  is central,  then $X_\E$ is connected
  and   the group $\pi_1(X_\E, A_0)$ is isomorphic to   $F/N$ where $F$ is the group defined in Section \ref{grF} and $N $
  is the normal subgroup of $F$
  generated  by the set
  $\{x_{A,B}\}$ where $A,B$ run over all  elements  of $ \E$ such   that $ A$ is an ascendent of $ B $. \end{lemma}

\begin{proof}    Pick  $A\in \E$ and represent the hull of $\{A_0, A\} $  by a based cube $(A_0,S)$ where $S=\{s_1,...,s_n\}$ is a cubic set of glides with $n \geq 0$.   As we know,  $A=[S]A_0=s_1 \cdots s_n A_0$.
For $k=1,..., n$, set $A_k=s_{1}\cdots s_k A_0  $.  The
  cube in $\E$ represented by the based cube $(A_0,
\{s_1,..., s_k\}   ) $ contains $ A_0, A_k $ among its vertices and has no
proper faces with the same property. Therefore this cube is the hull
of $\{A_0, A_k\}$ in~$\E$. Hence, $A_{k-1} $ is an ascendent of $
A_{k } $ for all $k$.

    The edges of $X_\E$ connecting consecutively  $A_0, A_1,..., A_n=A$
    form   a path $\rho_A$   from $A_0$ to $ A $.    The homotopy class of $\rho_A$
     in $X_\E$, obviously,     does not depend on the order in    $S $.
     Recall the    homomorphism $ \phi:\pi_1(X_\E, A_0)\to F$ from Lemma~\ref{phi--}.  By definition,
        $\phi
(\rho_A )=  x_{A_0, A_{1}} \, x_{A_1, A_{2}} \cdots x_{A_{n-1},
A_{n}} $. All generators on the right-hand side belong  to $N$
because   $A_{k-1} $ is an ascendent of $ A_{k } $ for all $k$.
Hence, $\phi (\rho_A )\in N$ for all $A\in \E$. Note that $\rho_{A_0}$ is the constant
path at $A_0$.

The existence of the  paths $\{\rho_A\}_A$ implies that $X_\E$ is
connected. As in   Lemma \ref{phi--},  these paths determine a
homomorphism $\psi:F\to \pi_1(X_\E, A_0)$ by
 $\psi(x_{A,B})=[\rho_A  \overline  {AB} \rho_B^{-1}]$.
  Observe that $N\subset \Ker \, \psi $. Indeed, for any $A,B\in \E$ such   that $ A$ is an ascendent of $ B $, we have
     $\rho_B=\rho_A   \overline  {AB} $ and
 $$\psi(x_{A,B})=[\rho_A  \overline  {AB} \rho_B^{-1}]= [\rho_A  \overline  {AB} \, \overline  {AB}^{-1}\rho_A^{-1}]=1.$$

We claim that the homomorphisms $\tilde \phi: \pi_1(X_\E, A_0) \to
F/N$ and $\tilde \psi: F/N\to \pi_1(X_\E, A_0)$ induced by $\phi$
and $\psi$ respectively, are mutually inverse  isomorphisms.
Clearly, $\tilde \psi \tilde \phi =\psi\phi =\id$. To check  that
$\tilde \phi \tilde \psi=\id$,  pick any $A,B\in \E$ such that
$AB^{-1}$ is a glide.   Denote the projection of $x_{A,B}\in F$ to
$F/N$ by $\tilde x_{A,B}$. We have  $$\tilde \phi \tilde \psi
(\tilde x_{A,B} )= \tilde \phi \psi (x_{A,B})=\tilde \phi ([\rho_A
\overline {AB} \rho_B^{-1}]) = \tilde \phi  (\rho_A ) \, \tilde
x_{A,B} \, \tilde \phi (\rho_B )^{-1}=  \tilde x_{A,B} $$ where at the last step  we
use that $\phi(\rho_A), \phi(\rho_B) \in N$.  Therefore
$\tilde \phi \tilde \psi=\id$.
\end{proof}

\subsection{Proof of Theorem~\ref{phi-+}}\label{pro-1} Let $\Pi$ be the group defined by the generators
and relations in this theorem. The relation $y_{A_0,A}\,
y_{A,A}=y_{A_0,A}$ implies that $y_{A,A}=1$ for all $A\in \E$. Then
  $y_{A,B}\, y_{B,A}=y_{A,A}=1$   for all $A, B\in \E$ so that $y_{B,A}=y_{A,B}^{-1}$. In particular,
$y_{A,A_0}=y_{A_0,A}^{-1}=1$ for all $A\in \E$.

The assumptions of Theorem~\ref{phi-+} imply
that
  $A_0\in \E$ is central. By    Lemma~\ref{phi-+c}, to prove the first part of the theorem,  it
suffices to construct an isomorphism $ \beta: F/N\to \Pi$. We define
$\beta$ on the generators by $\beta(x_{A,B})=y_{A,B}$.   The
compatibility with the relation $x_{A,B} \, x_{B,A}=1$ in $F$ is
clear. The compatibility with the relation \eqref{pathpp} follows
from the equalities
$$y_{A,B} \,  y_{B, C}\,  y_{C,D} \, y_{D,A} =y_{A, C}\,  y_{C,D} \, y_{D,A}=y_{A,  D} \, y_{D,A}=1.$$
Finally, if $A\in \E$ is an ascendent of $B\in \E$, then $A,B,A_0 $
lie in a cube in $\E$ and so $\beta(x_{A,B})=y_{A,B}=y_{A,A_0}\,
y_{A_0,B}=1$.

Next, we construct a homomorphism $ \gamma: \Pi \to  F/N $. For
$A,B\in \E$, let $Q=Q(A,B) $ be the   hull of $\{A, B\}$  in $\E$.
We can connect $A$ to $B$ by a path formed by     edges of $Q$ and
visiting consecutively   vertices $A=A_1, A_2,..., A_n=B $ of $ Q$.
Any two such paths are related by homotopies pushing the path across
square faces of $Q$. The relations in $F$ ensure that
$\gamma(y_{A,B})=\prod_k x_{A_{k-1}, A_k}$ is a well defined element
of $F/N$. If $A,B,C\in \E$ are vertices of a cube in $  \E$, then the
hulls   $Q(A,B) , Q(A,C), Q(B,C)$ are faces of this cube, and an argument
similar to the one above shows that $\gamma(y_{A,B}\,
y_{B,C})=\gamma(y_{A, C})$. Finally, a path connecting   $A_0$ to
$A$ may be chosen so that each vertex is an  ascendent of the next
vertex. By   definition of $N$, this gives $\gamma(y_{A_0,A})=1$.

For each generator $x_{A,B}$ of $F/N$, there is a 1-cell $P$ of $X_\E$
connecting $A$ to~$B$. Clearly, $P$ is a 1-cube in $\E$ with   vertices   $ A,B $.
No proper face  of $P$ contains  both $A$ and $B$. Hence,
$Q(A,B)=P$ and  $\gamma \beta
(x_{A,B})=\gamma(y_{A,B})=x_{A,B}$. Therefore $\gamma
\beta=\id$.

For any $A,B\in \E$, we have (in the notation above)
$$\beta \gamma(y_{A,B})=\prod_k \beta (x_{A_{k-1},
A_k})=  \prod_k y_{A_{k-1}, A_k}=y_{A,B}$$ where the last equality
follows from the defining relations in $\Pi$. Thus, $\beta
\gamma=\id$.

The second claim of the theorem follows from the definitions of  $\gamma$ and $\tilde \psi$ above.

\subsection{The fundamental groupoid} We now compute the fundamental groupoid $\pi=\pi_1(X_\E, \E)$  of the pair $(X_\E, \E)$. Recall that a groupoid is a (small) category in which all morphisms are isomorphisms. The   groupoid $\pi $    is the category whose objects are elements of $\E$ and whose morphisms are homotopy classes of paths in $X_\E$ with endpoints in $\E$. In the next assertion, we write down composition of morphisms in $\pi$ in the order opposite to the usual one; this makes composition in~$\pi$ compatible with multiplication of paths.



   \begin{corol}\label{corphi-+}   Under the conditions of Theorem~\ref{phi-+},
   the   groupoid $\pi=\pi_1(X_\E, \E)$ has generators $\{z_{A,B}:A\to B\}_{A,B\in \E}$ subject to the defining  relations
  $z_{A,C}=z_{A,B}\,  z_{B,C}$ for each triple   $A,B,C $ of vertices of a  cube in $\E$.  The image of each $
z_{A,B}$ in $\pi_1(X_\E, \E)$
  is represented by a path from $A$ to $B$ in the hull of $\{A,B\}$. \end{corol}

\begin{proof} Let $Z$ be the groupoid whose set of objects is $\E$ and whose morphisms are defined by the generators and relations in  this corollary (and are invertible).
It follows from the   relations that $z_{A,A}=1$ and $z_{A,B}^{-1}=z_{B,A}$ for all $A,B\in \E$.
Consider the functor $f:Z\to \pi $ which is the identity on the objects and which carries each   $z_{A,B}$ to the homotopy class of   paths from $A$ to $B$ in the hull of $\{A,B\}$. Since each path between 0-cells of $X_\E$ is homotopic to a path in the 1-skeleton of $X_\E$, the functor $f$ is surjective on the morphisms. To prove that $f$ is injective it is enough to prove that for any $A_0\in \E$, the   map $$f_{A_0}:\End_Z(A_0 )\to \End_\pi (A_0 )=\pi_1(X_\E,A_0)$$ induced by $f$ is injective. Theorem \ref{phi-+} implies that the formula $$g(y_{A,B})=z_{A_0,A} z_{A, B}z_{B,A_0}\in \End_Z(A_0)$$ defines a homomorphism  $g: \pi_1(X_\E,A_0) \to \End_Z(A_0)$.  It is easy to check that $f_{A_0} g=\id$ and  $g$ is onto.    Hence,  $g$ and $f_{A_0}$ are bijections.
\end{proof}

\section{Set-like glidings}\label{lloccgccgcgl}

We   study a class of   gliding systems generalizing
Example~\ref{exam}.5.

\subsection{Set-like gliding systems}\label{llossetgcgl} Consider a set $E$ and
its power group    $G=2^E \cong (\ZZ/2\ZZ)^E$. A gliding
system in $G$ is {\it set-like} if any   independent glides
in this system are disjoint as subsets of $E$. Thus, a
set-like gliding system in $G$ consists of  a  set $\G\subset
G\setminus \{1\}$ and a symmetric  set $\I\subset \G\times \G$ such
that $(s,t)\in \I \Longrightarrow s\cap t=\emptyset$.
For instance, the gliding system in Example~\ref{exam}.5 is
set-like.

\begin{lemma}\label{glidi-} A  set-like gliding system
is regular  in the sense of Section~\ref{regglsy}.
\end{lemma}

\begin{proof} We must
show that any finite set of independent glides $S\subset G$ is
cubic. Let $T_1,T_2$ be distinct subsets of $S$. Assume for
concreteness that there is $t\in T_1 \setminus T_2$.  Since all
glides in $S$ are disjoint as subsets of $E$,  we have
  $t\subset [T_1]$ and  $t\cap [T_2] =\emptyset$ where $[T_i]
=\cup_{s\in T_i} s $ for $i=1,2$. Thus,
$[T_1]\neq [T_2]$. Hence $S$ is cubic.
\end{proof}

   In the rest of this
section we discuss other   properties of set-like gliding
systems. Fix from now on a set $E$ and a set-like gliding
system in $G=2^E$.

\subsection{Orientation}\label{orisetthe} We claim that all subsets of $G$ are orientable in the sense of
Section~\ref{prelimArtin++}. It suffices to orient the set $ G$
itself
 because an orientation on $G$ induces  an orientation on any subset of~$G$. Pick an element $e_s\in s$ in every glide $s\subset E$.
 A 1-cell of $X_G$
relates two 0-cells $A,B\subset E$ such that $
AB=(A\setminus B)\cup (B\setminus A)$ is a glide in $G$.  Then
  $e_{AB} \in AB$ belongs to precisely one of the sets $A,B$. We orient the 1-cell in question towards
  the 0-cell containing $e_{AB}$. It is easy to check that this defines an orientation on
  $G$. (Here we need to use the assumption that independent glides   are
  disjoint.)

  An orientation on $G$ determines a   homomorphism $\mu_A:\pi_1(X_G, A)\to \A (G)$ for all $A\in  G$, see Section~\ref{prelimArtin++}.
  If   $E $ is   finite, then  the right-angled Artin  group $\A(G) $ is    finitely generated,      $\mu_A$ is   injective (Corollary \ref{simpleNEW++-eee}), and therefore  the group
   $\pi_1(X_G, A)$ is  biorderable, residually nilpotent, and
 embeds in $SL_n(\ZZ)$ for some~$n$.

\subsection{The evaluation map}\label{orisetthe2} Let $I^E$ be the set of   maps $E\to I$ viewed as the product of
copies of $I$ numerated by elements of $E$ and provided with the
product topology.    Assigning to each   set $A \subset E$   its
characteristic function $\delta_A: E\to \{0,1\}\subset I$ we obtain
a map from the 0-skeleton of $X_G$ to $  I^E$. We extend this map
  to $ X_G$ as follows. For a point $a\in X_G$
represented by a triple $(A,S, x:S\to I)$ as in
Section~\ref{cubecomplexX}, let $\omega(a):E\to I$ be the map
defined by
\begin{equation}\label{func}
\omega (a) (e) = \left\{
    \begin{array}{ll}
        x(s) & \mbox{for} \,\,   e\in s \setminus A\,\,  \mbox{with}  \,\,   s\in S\\
        1-x(s) & \mbox{for} \,\,   e\in s \cap A \,\,  \mbox{with}  \,\,   s\in S\\
         \delta_A(e) & \mbox{for all other} \,\, e \in E \, .
    \end{array}
\right.
\end{equation}
This definition makes  sense because distinct $s\in S$ are
disjoint as subsets of $E$.  The definition of $\omega (a)$ may be
rewritten in terms of characteristic functions:
$$\omega  (a)  = \delta_A   + (1-2\delta_A   )  \sum_{s\in S} x(s) \delta_{s} .$$  It is easy to check that the formula
$a\mapsto \omega_a$ yields a well-defined continuous map
$\omega:X_G\to I^E$ whose restriction   to
any cube in $X_G$ is an embedding. We call $\omega$ the {\it evaluation map}.


\begin{lemma}\label{Wisecubi}  Let $\E\subset  G=2^E$   and the following two conditions are met:

(i) $[S_1]=[S_2] \Longrightarrow S_1=S_2$ for any cubic sets of
glides $S_1,S_2\subset G$;

(ii)   any glide $s\subset E$ has a partition into two non-empty subsets  such that if $A \in \E$ and $sA \in \E$,
 then $s\cap A$ is one of these subsets.

\noindent Then the restriction of $\omega :X_G\to I^E$  to $X_\E\subset X_G$ is an injection.
   \end{lemma}

   \begin{proof} Suppose that
   $\omega (a_1)=\omega (a_2)$ for  some $a_1, a_2\in X_\E$. Represent each $a_i$ by
  a triple  $(A_i\in \E,S_i,x_i:S_i\to I)$.  Passing, if
   necessary, to  a face of the   cube, we can assume that
   $0<x_i(s)<1$ for all $s\in S_i$.  Set $f_i=\omega(a_i):E\to I$. Clearly,   $  f_i^{-1}( (0,1))=\cup_{s\in S_i} s=[S_i]
   $. The
   equality $f_1 =f_2 $ implies that
   $[S_1]=[S_2]$. By condition (i), $S_1$=$S_2$. Set $S=S_1=S_2$ and $f=f_1=f_2:E\to I$.

   We prove below that for
    any  glide $s \in S$, either $(\ast  )$ $s\cap A_1=s\cap A_2$ and $x_1(s)=x_2(s)$ or $(\ast \ast)$ $s\cap A_1=s\setminus A_2 $ and
    $x_1(s)=1-x_2(s)$. For each $s$ of the second type,    replace $(A_2, S, x_2)$ with    $(A'_2=sA_2, S, x'_2)$  where $x'_2:S\to
   I$ carries $s$ to
   $1-x(s )$ and  is equal to $x_2$ on $S\setminus \{s\}$. The triple $(A'_2, S, x'_2)$ represents the same point
   $a_2 \in \E$ and satisfies    $s\cap A_1=s\cap A'_2 $ and
   $x_1(s)=x'_2(s)$. Since different $s\in S$ are  disjoint as
   subsets of $E$, such replacements along various $s $ do not interfere with  each other
   and commute. Effecting them for all $s $ of the second type, we
 obtain a new triple $(A_2, S, x_2)$ representing
   $a_2 $ and satisfying  $s\cap A_1=s\cap A_2$,
   $x_1(s)=x_2(s)$ for all $s\in S$. Then $x_1=x_2:S\to I$. Also,
   $$A_1\cup \cup_{s\in S} s= f^{-1}((0,1])  =A_2\cup \cup_{s\in S} s.$$
   Since different $s\in S$ are disjoint as subsets of $E$ and satisfy $s\cap A_1=s\cap A_2$, we deduce that
    $A_1=A_2$. Thus, $a_1=a_2$.


It remains to prove that for
    any   $s \in S$, we have either $(\ast  )$   or $(\ast \ast)$.
    By \eqref{func},  the
   function $f :E\to I$   takes at most two  values on
   $s\subset E$.
   Suppose first that $f\vert_s$   takes two distinct values  (whose sum is   equal to $1$).
   By
   \eqref{func},  $f$ is constant on both   $s\cap A_i$   and  $s\setminus A_i$ for   $i=1,2$. If
 $f(s\cap A_1)=f(s\cap A_2)$,  then we have  $(\ast)$. If
 $f(s\cap A_1)= f(s\setminus A_2)$,  then we have  $(\ast\ast)$.   Suppose next
   that $f$ takes only one value on $s$. Condition (ii) implies that $s\cap A_1\neq \emptyset$ and $s\setminus A_1\neq \emptyset$.
   Since $f=\omega(a_1)$ takes only one value on $s$, formula \eqref{func} implies that this  value    is  $1/2$.
   Thus,
$x_1(s)=x_2(s)=1/2$. Applying condition (ii) again, we obtain  that either $s\cap
A_1=s\cap A_2$ or $s\cap A_1=s\setminus A_2 $. This gives,
respectively, $(\ast)$ or $(\ast \ast)$.
  \end{proof}


 \section{The dimer complex}\label{dimergroups}

\subsection{Cycles in graphs}\label{Cycles in graphs}   By a {\it graph} we   mean a 1-dimensional
CW-complex without loops (= 1-cells with coinciding endpoints).   The 0-cells and   (closed) 1-cells of a graph are called
 {\it vertices} and  {\it edges}, respectively. We allow multiple edges with the same vertices.
A set $s$ of edges  of a graph  $\Gamma $ is {\it cyclic}
  if $s$ is   finite and   each vertex of $\Gamma$   either is incident to two distinct edges belonging to $s$ or is not incident at all to edges belonging to $s$.   Then, the union, $\underline s$, of the  edges belonging to   $s$   is a subgraph of $\Gamma$ homeomorphic to a disjoint union of a finite number of circles.   A cyclic set of edges $s$ is a {\it cycle} if $\underline s$ is a single circle.
 A cycle  is {\it even} if   it   has an even number of elements.  An even cycle $s $   has a partition into two   subsets  called the {\it halves}  obtained by following along the   circle $\underline s$ and collecting all odd-numbered edges into one subset and all even-numbered edges into  another subset. The edges belonging to the same half have no common endpoints.



\subsection{The gliding system}\label{theglidingsystemcycles} Let   ${\Gamma}$ be a graph with   the set of edges
$E$ and let $G=G(\Gamma)=2^E$ be the power group   of $E$. Two  sets   $s,
t  \subset E$ are {\it independent} if     the edges belonging to $s$ have no
common vertices with the edges  belonging to $t$. This condition implies that $  s \cap
t =\emptyset$.

\begin{lemma}\label{glidi}  The even cycles   in ${\Gamma}$ in the role of glides together with the
independence relation   above form   a
  regular set-like gliding system in   $G$.
\end{lemma}

\begin{proof}   All conditions on a gliding system are straightforward.
The resulting gliding system is obviously set-like  in the sense of
Section~\ref{llossetgcgl}. It is regular   because by
Lemma~\ref{glidi-}, all set-like gliding systems are regular.
\end{proof}

Recall that for  a set   $A\subset E$ and an even cycle $s$ in
${\Gamma}$, the gliding of $A$ along $s$ gives $sA=(s\setminus A)\cup
(A\setminus s) $. We view $sA$ as the set obtained from   $A $
be removing the edges belonging to $s\cap A$ and adding all the
other edges of~$s$.

By Section~\ref{cubecomplexX}, the gliding system of Lemma \ref{glidi}
determines a cubed complex $ X_G $ with $0$-skeleton $G $.
By Corollary~\ref{simpleNEWBIS++++}, this complex   is
nonpositively curved.





\subsection{Dimer coverings}\label{newDimer coverings} Let   ${\Gamma}$ be a graph with   the set of edges
$E$.   A {\it dimer covering}, or a {\it perfect matching}, on   ${\Gamma}$ is a
subset   of $E$ such that every vertex of ${\Gamma}$ is incident to
exactly one edge of ${\Gamma}$ belonging to this subset.
 Let ${\D}=\D({\Gamma})\subset G=2^E$
be the set (possibly, empty) of all dimer coverings on ${\Gamma}$.

We provide  the  power group $G$
with the gliding system of Lemma~\ref{glidi}. By
Section~\ref{cubecomplexX}, the set  $\D\subset G$ determines a
cubed complex $X_{\D}  \subset X_G$ with
0-skeleton~${\D}$. We call $X_{\D}$ the {\it dimer complex} of
${\Gamma}$. By definition, two dimer coverings $A, B\in \D$ are connected by an edge in $X_\D$ if and only if $ AB \subset E$ is an even cycle.   Note   that if $s=AB$ is a cycle then it is necessarily   even
with complementary halves $s\cap A=A\setminus B$ and $s\cap B=B\setminus  A$.

\begin{lemma}\label{dimersMM} The set of dimer coverings
 ${\D}\subset G$  satisfies  the 3-cube condition of   Section \ref{cubecomplexX} and the square condition     of Section \ref{incl++}.
\end{lemma}

\begin{proof} 
The 3-cube condition follows from the square condition  which says
that for any $A\in {\D}$ and any independent even cycles $s,t$ in
${\Gamma}$ such that $sA, tA\in {\D}$, we   have $stA\in
{\D}$.
The inclusions $A, sA \in {\D}$ imply that    $s\cap A$ and $s\setminus A$ are the  halves of $s$. A similar claim holds for $t$. The independence of $s, t$ means that
  $s,t$ are
disjoint and   incident to disjoints sets of vertices.
The set   $stA\subset E$ is obtained from   $A $ through  simultaneous
replacement of the half $s\cap A$ of $s$ and the half $t\cap A$ of $t$ with the complementary halves.   It is clear   that $stA$ is a dimer
covering.
\end{proof}

 \begin{theor}\label{sspeccaseW} The dimer complex $X_{\D}$  is
nonpositively curved.
\end{theor}

This theorem follows from Lemma~\ref{dimersMM} and Corollary
\ref{simpleNEWBIS}.

\subsection{Examples}\label{rerere} 1. Let ${\Gamma}$ be a  triangle (with 3   vertices and 3
edges). The  set of glides in $G=G(\Gamma)$  is empty, $X_G  =G$ consists of 8 points, and $X_\D  =\D  =\emptyset$.

2. Let ${\Gamma}$ be a  square  (with 4   vertices and 4
edges). Then $G=G(\Gamma)$ has   one glide, $X_G  $ is a disjoint union of 8  closed intervals, and $X_\D  $ is one of them.

3.   More generally, let ${\Gamma}$ be formed by $n\geq 1$ cyclically connected vertices and $\D=\D(\Gamma )$. If $n$ is odd, then     $X_\D =\D =\emptyset$. If $n$ is even, then ${\Gamma}$ has two dimer
  coverings   and      $X_\D $
  is a segment.

4. Let ${\Gamma}$ be  formed by 2 vertices
  and $3 $   connecting them edges. Then $G=G(\Gamma)$ has 3 glides and $X_{G} $ is a disjoint union of two complete graphs on 4 vertices. The space $X_\D $ is formed by 3 vertices and 3 edges of one of these   graphs.

 5. More generally, for   $n\geq 1$, consider the graph ${\Gamma}^n$ formed by 2 vertices
  and $n $   connecting them edges. A dimer covering of ${\Gamma}^n$ consists of a single edge, and so, the set $\D=\D(\Gamma^n)$ has $n$ elements. The graph ${\Gamma}^n$ has $n(n-1)/2$
 cycles, all  of length 2 and none of them independent.    The complex $X_{\D} $ is
a complete graph on $n$ vertices.


   \subsection{Remark}\label{rexexexexsere} If ${\Gamma}$ is a disjoint union of graphs $ {\Gamma}_1$, $ {\Gamma}_2$, then $G=G(\Gamma)=G_1\times G_2$, where $G_i=G(\Gamma_i)$ for $i=1,2$, and
  $X_G   = X_{G_1}   \times X_{G_2}  $. If the graphs $ {\Gamma}_1$, $ {\Gamma}_2$ are finite, then $\D=\D(\Gamma)= \D_1 \times \D_2$, where $\D_i=\D(\Gamma_i)$ for $i=1,2$, and  $X_\D=X_{\D_1}\times X_{\D_2}$.

\section{The dimer  group}\label{newDimer coverings+}

We focus  now on the case of finite graphs, i.e., graphs with finite sets of edges and vertices. Throughout  this section,   ${\Gamma}$ is a finite graph with set of edges  $E$. Set $G=G(\Gamma)=2^E$ and $\D=\D(\Gamma)\subset G$.

\subsection{The   group $D(\Gamma)$}\label{verynewnewDimer coverings+}
 The finiteness of $\Gamma$ implies that the cubed complexes $X_G$ and $X_{\D}\subset X_G$ are finite CW-spaces.

\begin{lemma}\label{dimersMMHULL} Every 2-element subset of
${\D} $  has a hull in ${\D}$ (cf.\ Section~\ref{mincu}).
\end{lemma}

\begin{proof}  Let $A,B \in \D $ be      dimer coverings of $\Gamma$. Then
  $AB=(A \setminus B) \cup (B\setminus
A)$ is a cyclic set of edges. It splits  uniquely as  a   union of    independent   cycles. Denote the
set of these cycles by $S$. All cycles in $S$ are even: their halves are their intersections   with $A \setminus B$ and   $B \setminus A$.  In the notation of Section \ref{BU}, $[S]=AB$. As in the proof of Lemma~\ref{dimersMM}, one easily shows that
for all $ T\subset S $, the set $[T] A\subset E$ is a dimer covering
of~$\Gamma$. The based cube $(A,S)$ determines   a cube, $Q$, in
${\D}$ whose set of vertices contains   $A$ and  $B=[S]A$. An arbitrary cube $Q'$ in
$\D$ whose set of vertices contains   $A$ and  $B$ can be represented   by a based cube
$(A,S')$ such that   $B=[T] A$ for some $T\subset S'$. We
have $[T]= BA^{-1} =AB $. Since a cyclic set of edges
splits as a   union of independent cycles in a unique way, $T=S$. So,  $Q$ is a face of $Q'$.
\end{proof}

Lemma~\ref{dimersMMHULL} implies  that   $X_{\D}$    is connected, and
Theorem~\ref{sspeccaseW} implies that
$X_{\D}$ is aspherical. We call the  fundamental
group of $X_\D$ the {\it dimer group} of ${\Gamma}$ and denote it $D({\Gamma})$. We refer to Section~\ref{Cubed complexes} for properties of $D(\Gamma)$ which follow from Theorem~\ref{sspeccaseW}.
If ${\Gamma}$ does not have dimer coverings, then $X_\D=\emptyset$ and $D({\Gamma})=\{1\}$.

In Examples~\ref{rerere}.1-3, $D(\Gamma)=\{1\}$. In~\ref{rerere}.4, $D(\Gamma)=\ZZ$.
In~\ref{rerere}.5, the group $D({\Gamma}^n)$ is a free group of rank $(n-1)(n-2)/2$. By Remark~\ref{rexexexexsere}, if $\Gamma$ is a disjoint union of finite graphs $\Gamma_1$, $\Gamma_2$, then $D(\Gamma)=D(\Gamma_1)\times D(\Gamma_2)$.
Therefore   any free abelian group  of
 finite rank can be realized  as the  dimer group of a finite graph (see also Remark~\ref{remsmsms}.1). Other
 abelian groups   cannot be realized as dimer groups
 because  all dimer groups are finitely generated and torsion-free.

 \begin{theor}\label{s++1} The
 inclusion homomorphism $D({\Gamma})= \pi_1(X_{{\D}}, A) \to \pi_1(X_G, A)$ is
 injective for all $A \in  {\D}$.
\end{theor}

   Theorem~\ref{s++1} follows from Corollary~\ref{simpleNEW++-} and Lemma~\ref{dimersMM}.

   \subsection{Typing homomorphisms}\label{verynewnewDimer coverings++++++++} By Section~\ref{orisetthe}, a choice of an element $e_s\in s$ in each even cycle $s$  in ${\Gamma}$  determines an orientation on $G$  which, in its turn,  determines  a   homomorphism $\mu_A:\pi_1(X_G, A)\to \A  $ for all $A\in  G$.
 Here $\A=\A(G)$ is the right-angled Artin  group associated with the gliding system in  $G$.
Since   $E $ is  finite, the group           $\A $ is    finitely generated and $\mu_A$ is an injection. Composing $\mu_A$
with the inclusion homomorphism $\pi_1(X_{{\D}}, A) \to \pi_1(X_G, A)$ we obtain the {\it typing homomorphism}  $\pi_1(X_{{\D}}, A)\to \A $ also denoted  $\mu_A$. The same homomorphism is obtained by restricting the orientation on $G$ to $\D$ and considering the associated typing homomorphism as   in Section~\ref{prelimArtin++}.  The   homomorphism $\mu_A: \pi_1(X_{{\D}}, A)\to \A $  embeds the dimer group into a finitely generated right-handed Artin group. We refer to Section~\ref{Applications} for properties of the dimer group which follow from this fact.

Though we shall not need it in the sequel, we state here a few    facts concerning the    typing homomorphisms. First of all,  two families $\{e_s\in s\}_s$ and $\{e'_s\in s\}_s$    determine the same orientation  on $ \D$ if and only if   $e_s, e'_s$  belong to the same half of~$s$ for each even cycle  $s$ in $\Gamma$. Thus, the orientations  on $\D$ associated with such families are fully determined by a choice of a   half in each   $s$. If the distinguished half of a cycle    $s_0$  is replaced with the complementary half,
the typing homomorphism $\mu_A: \pi_1(X_{{\D}}, A) \to \A$ is replaced by  its composition with the automorphism of $\A$ inverting the generator $g_{s_0}\in \A$  and fixing  the generators $g_s\in \A$ for all $s\neq s_0$. Thus, $\mu_A$ does not depend on the  choice of the   halves   up to composition with   automorphisms of $\A$. Finally, we point out a non-trivial homomorphism $u$ from $\A$ to another group  such that $   \Ima \mu_A\subset \Ker u$. Let $\B$ be the  right-angled Artin group
  with generators $\{h_e\}_{e \in E }$ and    relations $h_e h_f
  h_e^{-1} h_f^{-1}=1$ for all edges $e,f\in E$ having no common vertices. Choose in each even cycle $s$ in $\Gamma$   a half $s'\subset s$. The formula $u(g_s)= \prod_{e\in s\setminus s'} h_e ^{-1} \prod_{e\in s'} h_e  $ defines a homomorphism $u:\A\to \B$. We claim that $u\mu_A=1$   where $A\in \D$ and $\mu_A: \pi_1(X_{{\D}}, A)\to \A $ is the typing homomorphism determined by the orientation on $\D$  associated with the   family $\{s'\}_s$. To see this, we  compute   $\mu=\mu_A$ on a path~$\alpha$ in~$X_\D$ using~\eqref{pathee}. Observe that if $A_k\in \D$ is the terminal endpoint   of
$e_k$ and the initial endpoint of $e_{k+1}$, then for all $k$,  $$u(g_{\vert e_k\vert}^{\nu_k})= \prod_{e\in A_{k-1}} h_e ^{-1} \prod_{e\in A_k} h_e.$$
 Multiplying these formulas over all $k$,  we obtain $u\mu_A(\alpha)= \prod_{e\in A_{0}} h_e ^{-1} \prod_{e\in A_n} h_e$. For $ A_0=A_n=A$,   this gives $u\mu_A(\alpha)=1$.


 \subsection{Generators and relations}\label{presbygres+} For a dimer covering  $A\in \D$ and a vertex~$v$ of~${\Gamma}$, denote the only edge of   $A$ incident to $v$ by $A_v$.    A triple $A,B,C\in \D$ is {\it flat} if for any vertex $v$ of ${\Gamma}$, at least two of the edges $A_v, B_v, C_v$ are equal.

\begin{theor}\label{phi-+VBBB}    For each $A_0\in
\D
  $,
   the dimer group $ \pi_1(X_\D, A_0) $ is isomorphic to  the group with generators $\{y_{A,B}\}_{A,B\in \D}$ subject to the following relations:
  $y_{A,C}=y_{A,B}\,  y_{B,C}$ for each  flat triple   $A,B,C \in D$ and $y_{A_0,A}=1$ for all $A\in \D$. \end{theor}

\begin{proof} This theorem follows from Lemma \ref{dimersMMHULL},  Theorem \ref{phi-+}, and    the  following claim: three dimer coverings $A,B,C \in \D$ are
   vertices of a cube in $\D$ if and only if the triple $A,B,C$ is flat. We now prove this claim.

   Suppose   that the triple $ A,B,C $ is flat.
Then
 every vertex of
   ${\Gamma}$ is incident to a unique edge belonging to  at least  two   of the
sets $A,B,C\subset E$. Such edges form a dimer covering, ${K}$,   of~${\Gamma}$.
For any vertex $v$ of ${\Gamma}$,   either   $A_v=B_v=C_v ={K}_v $ or  $A_v=B_v ={K}_v\neq C_v $  up to permutation of
   $A,B,C$. In the first case,
  the sets $A{K}, B{K},
   C{K}\subset E$ contain no edges incident to $v$. In the second case,    $A{K},
   B{K}
 $ contain no edges incident to $v$ while $C{K}$ contains two such edges $C_v$ and ${K}_v$. This   shows that $A{K}, B{K}, C{K}$ are pairwise
   independent cyclic sets of edges. They split
 (uniquely) as     unions of independent cycles. All these cycles are even:  their intersections with $K$ are their halves.    Denote the resulting sets of
 even cycles by $X, Y, Z$, respectively. Thus,   $A{K}=[X], B{K}=[Y], C{K}=[Z]$. Equivalently, $A=[X]{K}$, $ B= [Y] {K}$, and $ C=[Z]{K}$. Then the cube in $\D$ determined
 by the based cube $({K}, X \cup Y\cup Z)$ contains $A,B$, and $C$.

 Conversely, suppose that $A,B,C \in \D$ are
   vertices of a cube in $\D$.  By the last remark of Section \ref{mincu},    there is ${K}\in \D$ and
a   set of independent even  cycles $S $ with a partition $S=X \amalg Y \amalg Z$ such
that $A=[X]{K}$, $B=[Y] {K}$, and $C=[Z]{K}$.
Consider a   vertex $v$ of ${\Gamma}$. If $v$ is not incident to the edges forming the cycles of $ S$, then $A_v=B_v=C_v ={K}_v $. If   $v $ is   incident to an edge   belonging to a cycle $s\in   X$, then the edges forming the cycles of $Y$ and $Z$ are not incident to~$v$  and
  $ B_v =C_v =K_v$. The cases $s\in Y$ and $s\in Z$ are  similar. In all cases, at least two of the edges $A_v, B_v, C_v$ are equal. Therefore, the triple $ A,B,C $ is flat.
\end{proof}

  Corollary  \ref{corphi-+}     yields the following computation of the {\it dimer groupoid} $ \pi_1(X_\D, \D)$.

  \begin{theor}\label{phdbdbdbBBB}   The   groupoid $ \pi_1(X_\D, \D)$ is presented by generators $\{z_{A,B}:A\to B\}_{A,B\in \D}$ and relations
  $z_{A,C}=z_{A,B}\,  z_{B,C}$ for every  flat triple   $A,B,C \in \D$. \end{theor}

 \subsection{The space $L(\Gamma)$}\label{spaceLgamma} We  relate  the dimer complex $X_\D \subset X_G$ to the space $L(\Gamma)$ of dimer labelings of $\Gamma$. In the sequel, non-even cycles in $\Gamma$ are said to be    {\it odd}.

  \begin{theor}\label{spaces++1} Let  $\omega:X_G\to I^E$ be the evaluation map from Section~\ref{orisetthe2}.  The restriction of  $\omega $ to $X_\D$ is an embedding   whose image is a path-connected component $  L_0(\Gamma) $ of   the space of dimer labelings $ L(\Gamma)\subset I^E$. The component $  L_0(\Gamma) $ contains all dimer labelings associated with   dimer coverings of ${\Gamma}$.
 \end{theor}

 \begin{proof} The map $\omega$ carries a dimer covering of ${\Gamma}$ into the associated dimer labeling of ${\Gamma}$. Therefore $\omega (X_\D)$ contains the set $\D$ viewed as a subset of $L (\Gamma)$.
  Since $X_\D$ is path connected, $\omega (X_\D) $ is contained in a path connected component, $L_0$,  of $L(\Gamma) $.   We shall show that the restriction of  $\omega $ to $X_\D$ is a homeomorphism onto $ L_0 $.

  Observe that $\E=\D\subset G$ satisfies  the conditions   of Lemma \ref{Wisecubi}.
  Condition (i) holds because a  subset of $E$ may
split as a   union of independent cycles in at most one way. The partition of  even cycles into halves satisfies (ii).   Lemma \ref{Wisecubi} implies that the restriction of  $\omega $ to $X_\D$ is injective.

Consider a   dimer labeling $\ell:E\to I=[0,1]$. Then  $\ell^{-1}((0,1))\subset E$  is a cyclic set of edges. It splits uniquely as a (disjoint) union of  $n \geq 0$ independent  cycles   $s_1,..., s_n$. If $s_i$ is odd for some $i$, then   $\ell(s_i)=\{1/2\}$ (otherwise, $s_i$ could be partitioned into edges with $\ell<1/2$ and edges with $\ell>1/2$ and would be even). Then, any  deformation of $\ell$ in $L({\Gamma})$   preserves    $\ell(s_i)$. So, $\ell \notin L_0$.
Suppose that all the cycles $s_1,..., s_n$ are even. We verify   that   $\ell \in \omega(X_\D)\subset L_0$. This will imply that $L_0=\omega(X_\D) $. For each $i=1,...,n$, pick a half, $s'_i$, of~$s_i$. The definition of a dimer labeling implies that $\ell$ takes the same value, $x_i\in (0,1)$,  on all edges belonging to $s'_i$ and the value $1-x_i$ on all   edges belonging to the complementary half $s_i \setminus s'_i$.   Set $A=\ell^{-1}(\{1\}) \cup \cup_{i=1}^n s'_i\subset E $. Clearly,  the edges belonging to $A$ have no common vertices.  Since each vertex of ${\Gamma}$ is incident to an edge with positive label, it is incident to an edge belonging to $A$. Therefore  $A\in \D$.
Since $s_1,..., s_n$ are independent even cycles and $A\cap s_i= s'_i$ for each $i$, all    vertices of the  based cube $(A, S= \{s_1,..., s_n\})$ belong to~${\D}$.  The triple $(A, S, x:S\to I)$, where $x(s_i)=1-x_i$ for all $i$, represents a point $a\in X_\D$. It is easy to check that  $\omega(a)=\ell$. So,  $\ell \in \omega(X_\D)$.  \end{proof}


\begin{theor}\label{dimlabelB}  The component $L_0(\Gamma)$ of      $L({\Gamma})$ is    homeomorphic to the dimer  complex of $\Gamma$.  All  other  components of      $L({\Gamma})$ are  homeomorphic to the dimer  complexes of certain subgraphs of $\Gamma$.\end{theor}

\begin{proof} The first claim   follows from Theorem~\ref{spaces++1}. Arbitrary components of $L({\Gamma})$ can be described as follows. Consider a  set $S$ of independent odd cycles in~${\Gamma}$. Deleting from~${\Gamma}$ the edges belonging to these cycles,  their vertices, and all the edges of~$\Gamma$ incident to these vertices, we obtain a subgraph ${\Gamma}_S$ of ${\Gamma}$. Each   dimer labeling of ${\Gamma}_S$ extends to a dimer labeling of ${\Gamma}$
  assigning $1/2$ to the edges belonging to the cycles in $S$ and $0$ to all other edges of ${\Gamma}$ not lying in ${\Gamma}_S$. This defines an embedding $i :L({\Gamma}_S)\hookrightarrow L({\Gamma})$.  The   proof of Theorem \ref{spaces++1} shows that the image of $i  $ is a  union of connected components of $L({\Gamma})$. In particular, $i (L_0({\Gamma}_S))$ is a component of $L(\Gamma)$. Moreover,    every component of $L({\Gamma})$ is realized as $i (L_0({\Gamma}_S))$ for a unique~$S$. (In particular,   $L_0({\Gamma})$   corresponds to $S=\emptyset$.) It remains to note that $ L_0({\Gamma}_S) $ is homeomorphic to the dimer complex of $\Gamma_S$.
\end{proof}

\begin{corol}\label{dimlabelB++}    All   components of      $L({\Gamma})$ are  homeomorphic to non-positively curved cubed complexes and  are aspherical. Their fundamental groups are isomorphic to the dimer groups of certain subgraphs of $ {\Gamma} $. \end{corol}

\subsection{Remarks}\label{remsmsms}  1.
Consider  finite graphs ${\Gamma}_1$, ${\Gamma}_2$
admitting dimer coverings. Let ${\Gamma}'$ be   obtained from   $\Gamma={\Gamma}_1 \amalg {\Gamma}_2$ by adding an edge
connecting a vertex of ${\Gamma}_1$ with a vertex of ${\Gamma}_2$. Then
 $D({\Gamma}')=D({\Gamma} )$ and $X_\D({\Gamma}')= X_\D({\Gamma} )$.
This implies that for   any  finite graph, there is a connected finite graph with the same dimer group.

2. If a finite graph  has no dimer coverings, then one can  subdivide some of its edges into two subedges so that the resulting graph has   dimer coverings.   Subdivision of edges into 3 or more  subedges is redundant. This is clear from the fact that  if  a graph ${\Gamma}'$ is obtained from   a finite graph ${\Gamma}$  by adding two new vertices inside the same edge, then there is a   canonical  homeomorphism $L({\Gamma})\approx  L({\Gamma}')$ carrying $L_0({\Gamma})$ onto $ L_0({\Gamma}')$ and $\D({\Gamma})$ onto $ \D({\Gamma}')$.

3. The typing homomorphism $\mu : D(\Gamma) \to \A $ of Section~\ref{verynewnewDimer coverings++++++++} induces an algebra homomorphism $\mu^\ast:H^\ast(\A)\to  H^\ast(D(\Gamma))$ which may provide non-trivial cohomology classes of $D(\Gamma)$ (with coefficients in any commutative ring). The algebra $H^\ast(\A)$  can be computed from  the    fact that the  cells  of the Salvetti complex  appear in the form of tori and so, the boundary maps in the cellular chain complex are zero    (see, for example,  \cite{Ch}).  The equality $   u \mu =1$ in Section~\ref{verynewnewDimer coverings++++++++} shows that $\mu^\ast$ annihilates   $u^*( H^\ast(\B))\subset   H^\ast(\A)$. It would be interesting to know whether $ \Ker \mu^\ast$ can be bigger than the ideal generated by $u^*( H^\ast(\B))$.

4. The proof of Theorem~\ref{dimlabelB}  shows that  connected   components of      $L({\Gamma})$  bijectively correspond to   sets of independent odd cycles in $\Gamma$.

\section{Dimers vs. braids}\label{extension-----}

The braid groups of a finite graph $\Gamma$ share many properties of the dimer group $D(\Gamma)$. They    are realizable as fundamental groups of non-negativelty curved   complexes, and they embed in a  finitely generated right-angled Artin group, see \cite{Ab}, \cite{CW}. We   construct   a family of homomorphisms from $D(\Gamma)$ to the braid groups of~$\Gamma$. We begin by recalling the   braid groups of CW-spaces.

\subsection{Braid groups}\label{confbraid} For a topological space $X$ and an  integer $n\geq 1$, the {\it ordered $n$-configuration space}  $\tilde C_n=\tilde C_n(X) \subset  X^n$ is formed by  $n$-tuples of pairwise distinct points of $X$. The symmetric group $S_n$ acts on $\tilde C_n $ by permutations of the coordinates, and the quotient  $C_n =\tilde C_n /S_n$ is the {\it unordered $n$-configuration space} of~$X$. If $X$ is a connected CW-space, then $C_n $ is path connected. The  group  $B_n (X) =\pi_1(C_n )$ is the {\it $n$-th braid group} of $X$. The covering $\tilde C_n \to  C_n $
determines (at least up to conjugation) a homomorphism $\sigma_n: B_n(X)\to S_n$.

\subsection{V-orientations}\label{V-orientations of cycles}  Given a cycle $s $ in a   graph $\Gamma$,  denote by $\partial s$ the set of vertices of  $s$, i.e.,  the set of vertices of $\Gamma$ incident to the edges belonging to $s$. Clearly, $\card (s)=\card (\partial s)$.     If   $s$ is even, then $\partial s$ has a partition into two   subsets    obtained by following along the   circle $\underline s$ and collecting all odd-numbered vertices into one subset and all even-numbered vertices into  another subset. These subsets are called   {\it vertex-halves} or, shorter, {\it v-halves} of $s$.
 A {\it v-orientation} of $\Gamma$ is a choice of a distinguished v-half in each even cycle in $\Gamma$.
  If $k$ is the number of even cycles in $\Gamma$, then $\Gamma$ has $2^k$ v-orientations.


\subsection{The  map $\Theta: L_0(\Gamma)\to C_N(\Gamma)$}\label{amapcgcgcg} Let $\Gamma$ be a v-oriented finite graph admitting at least one dimer covering.
All dimer coverings of $\Gamma$ have the same number, $N$, of edges equal to half of  the number of  vertices of $\Gamma$ (the latter number has to be even). Let $E$ be the set of edges of $\Gamma$.  We   define a map $\Theta$ from the dimer space $L_0(\Gamma)\subset I^E$ to the unordered $N$-configuration space $ C_N(\Gamma)$. To this end, we parametrize all edges of $\Gamma$ by the interval $I=[0,1]$.   This   allows  us to  consider  convex combinations of points of
an   edge.   The homotopy class of  $\Theta$ will  not depend on the  parametrizations.

 For a dimer labeling $\ell\in L_0(\Gamma)$, we define an $N$-point  set $\Theta(\ell)\subset \Gamma$ as follows. Each edge $e\in \ell^{-1}(\{1\}) \subset E$ contributes to $\Theta(\ell)$ the   mid-point  $(a_e+b_e)/2\in e$ where $a_e, b_e$ are the vertices of $e$. Other points of $\Theta(\ell)$ arise from the cyclic set of edges $\ell^{-1}((0,1)) \subset E$. This set is formed by several
independent  cycles $s_1,..., s_n\subset E$. The inclusion $\ell\in L_0(\Gamma)$ guarantees that   these cycles are even, cf. the proof of Theorem~\ref{spaces++1}. Each $s_i$ contributes $\card (s_i)/2$   points to $\Theta(\ell)$. If $\ell (s_i)=\{1/2\}$, then these points are the vertices of $s_i$ belonging to  the distinguished  v-half.  If $\ell (s_i)\neq \{1/2\}$, then
$\ell$ takes the same value    $  x_i \in (1/2, 1) $ on every second edge of $s_i$. Each such edge, $e$, has a vertex,   $a_e$, in the distinguished v-half of $s_i$ and a  vertex, $b_e$, in the complementary v-half. We include in $\Theta(\ell)$ the point of $e$ represented by the convex combination  $(3/2- x_i) a_e +( x_i-1/2) b_e $. The coefficients
are chosen so that  when $x_i$ converges to $1/2$ the combination converges to $a_e$, and  when $x_i$ converges to $1 $ the combination converges to   $(a_e+b_e)/2 $. Note that the points contributed by $s_i$ lie on the circle $\underline s_i\subset \Gamma$. Therefore all selected points are pairwise distinct  and   $\card (\Theta(\ell))=N$.
This defines a continuous map $\Theta: L_0(\Gamma)\to C_N(\Gamma)$.

\subsection{The homomorphism $\theta$}\label{amghohmommmngcg} Let $\Gamma$, $N$, $E$ be as in  Section~\ref{amapcgcgcg} with  connected~$\Gamma$. Then $ C_N(\Gamma)$ is connected and  the map $\Theta: L_0(\Gamma)\to C_N(\Gamma) $   induces a homomorphism of the fundamental groups $ \theta=\theta(\Gamma): D(\Gamma)\to B_N(\Gamma)$. This homomorphism  is not necessarily injective  and  may be trivial, see Example~\ref{EEEDDD+}  below.

The composition of    $\theta$ with   $\sigma_N: B_N(\Gamma)\to S_N$ can be explicitly computed as follows.   Pick a  dimer covering $A $ of $\Gamma$ and    mark  the edges belonging to  $A $ by (distinct) numbers $1,2,..., N$. A loop in $  X_\D (\Gamma)$ based at $A$  is represented  by a sequence of consecutive glidings of $A$ along certain even cycles $s_1,..., s_n$. Recursively in  $i=1,..., n$, we accompany the $i$-th  gliding  with the transformation  of the marked dimer covering which keeps the marked edges   not belonging to $s_i$ and pushes each marked edge  in $s_i$  through its vertex belonging to the distinguished v-half of $s_i$. After the $n$-th gliding, we obtain  the same dimer covering $A$ with a new marking. The resulting permutation of the set $\{1,..., N\}$  is the value of  $\sigma_N \theta:D(\Gamma)=\pi_1(X_\D (\Gamma), A) \to S_N$ on the loop.  The following example    shows that, generally speaking,  $\sigma_N \theta \neq 1$ and $\sigma_N \theta $ depends on the   v-orientation of the graph.

\subsection{Example}\label{EEEDDD}  Consider the   graph $\Gamma$  in Figure~\ref{fig1} with   vertices $a,b,c,d,e,f$. This graph has three cycles $s_1$, $s_2$, $s_{12}$ formed, respectively,  by the edges of the left square, by the edges of the right square, by all edges except the middle vertical edge. These cycles are even. We distinguish the  $v$-halves, respectively, $\{a , e\}$, $\{  c,e\}$, $\{b,d,f\}$. The vertical edges of $\Gamma$ form a dimer covering, $A$. The group $D(\Gamma)=\pi_1(X_\D (\Gamma), A)$ is an infinite cyclic group with generator $t$ represented by the sequence of glidings $A\mapsto s_1 A \mapsto s_{12} s_1 A \mapsto s_2s_{12} s_1 A$. To compute  $\sigma_3 \theta:D(\Gamma) \to S_3$, we  mark the edges of $A$ with   $1$, $2$, $3$ from left to right. The transformations  of~$A$ under the  glidings  are shown in Figure~\ref{fig1}. Therefore $\sigma_3 \theta(t)=(231)$.   The opposite choice of the distinguished v-half in  $s_2$ gives $\sigma_3 \theta(t)=(213)$.

\begin{figure}[h,t]
\includegraphics[width=9cm,height=5.4cm]{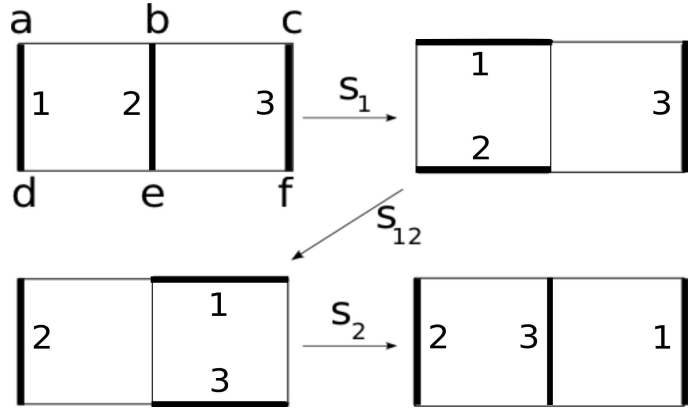}
\caption{Transformations of a marked dimer covering}\label{fig1}
\end{figure}

\subsection{The homomorphisms $\theta_n$}\label{amghohmommmngcgnnnn}     The homomorphism $\theta:D(\Gamma)\to B_N(\Gamma)$ of Section~\ref{amghohmommmngcg} can be included into a family of   homomorphisms numerated by   maps
  $n: E\to \{0,1,2,...\}$. Set $\vert n\vert =\sum_{e\in E} n(e)$.   Let $\Gamma_n$ be the graph obtained from $\Gamma$ by adding $2 n(e)$ new vertices inside   each  edge $e $.  The  canonical bijection $\D(\Gamma) \approx \D(\Gamma_n)$ (cf. Remark~\ref{remsmsms}.2)  extends to a    homeomorphism $X_\D({\Gamma}) \approx  X_\D({\Gamma_n})$ which, in its turn, induces a canonical isomorphism $D(\Gamma)\cong D(\Gamma_n)$. The v-orientation of $\Gamma$ induces a v-orientation of $\Gamma_n$: the distinguished v-half of a cycle in $\Gamma_n$ is the one including   the distinguished v-half of the corresponding cycle in $\Gamma$.   We define a homomorphism $\theta_n: D(\Gamma)\to B_{N+\vert n\vert}(\Gamma)$ as the composition
$$   D(\Gamma) \cong  D(\Gamma_n)\stackrel{\theta(\Gamma_n)}{ \xrightarrow{\hspace*{1cm}} } B_{N+\vert n\vert}(\Gamma_n)=B_{N+\vert n\vert}(\Gamma).$$
For $n=0$, we have  $\theta_n=\theta$.
Example~\ref{EEEDDD+} below exhibits v-oriented graphs such that $\theta_n=1$ for all $n$. It would be interesting to find out whether all $\theta_n$ may be trivial for all v-orientations of a graph with non-trivial dimer group. Is it true that $\cap\, \cap_n  \Ker (\sigma_{N+\vert n\vert} \theta_n )=1$ where the first intersection runs over all   v-orientations?

\subsection{Example}\label{EEEDDD+} Consider a   finite connected graph $\Gamma$ which is bipartite, i.e., the set of vertices of $\Gamma$ is partitioned into two subsets $V_0, V_1$ such that every edge has one vertex in each. All cycles in $\Gamma$ are even. We v-orient $\Gamma$   by selecting in every   cycle in $\Gamma$ the v-half formed by the vertices belonging to $V_0$. Then the image of $\Theta: L_0(\Gamma)\to C_N(\Gamma) $  lies in   $\prod_{v\in V_0} \Gamma_v\subset C_N(\Gamma)$ where  $\Gamma_v\subset \Gamma$ is the union of all half-edges adjacent to~$v$. Since $\Gamma_v$ is contractible, the map $\Theta$ is homotopic to a constant map. Then $\theta=1$. The graph $\Gamma_n$ is also bipartite for all $n$, and   $\theta_n=1$. Other   v-orientations in $\Gamma$ may give non-trivial $\theta$ and $\theta_n$, cf. Example~\ref{EEEDDD}.

 \section{Extension to hypergraphs}\label{extension}

We extend the   results of Section \ref{newDimer coverings+}   to hypergraphs.

\subsection{Hypergraphs}\label{hyper} A {\it hypergraph} is a triple $\Gamma=(E,V, \partial)$ consisting of two  sets $E$, $V$  and a map $\partial:E\to 2^V$ such that $\partial e\neq \emptyset$ for all $e\in E$. The elements of $E$ are  {\it edges} of $\Gamma$, the elements of $V$ are  {\it vertices} of $\Gamma$, and $\partial$ is  the {\it boundary map}. For $e\in E$, the elements of $\partial e\subset V$ are   {\it vertices   incident  to}~$e$ or, shorter,   {\it vertices of} $e$.

We briefly discuss examples of hypergraphs. A graph gives rise to  a hypergraph in the obvious way. Every matrix $M$ over an abelian group yields a hypergraph whose edges are   non-zero rows of $M$, whose vertices are   columns of $M$, and whose boundary map carries a row   to the set of   columns  containing   non-zero entries of this row. A CW-complex gives rise to a sequence of hypergraphs associated as above with the matrices of the boundary homomorphisms in the   cellular chain complex. Coverings of sets by   subsets also yield hypergraphs:   a set $E$ and a family of subsets $\{E_v\subset E\}_{v\in V}$ with  $\cup_v E_v=E$ determine a hypergraph  $(E,V, \partial)$ where $\partial e=\{v\in V\,\vert\, e\in E_v\}$ for any $e\in E$.

 \subsection{The gliding system}\label{cychyper}  Given a  hypergraph $\Gamma=(E,V, \partial)$, we call two sets $s,t\subset E$ {\it independent} if $\partial s\cap \partial t=\emptyset$. Of course, independent sets are disjoint.

  A {\it cyclic set of edges} in  $\Gamma $
  is a  finite set $s\subset E$ such that for every $v\in V $, the set $\{e\in s\,\vert\, v\in  \partial e\}$   has two elements or is empty. A cyclic set of edges  is a {\it cycle} if
 it does not contain smaller non-empty cyclic sets of edges.

  \begin{lemma}\label{cycliccersMM} If $s\subset E$ is a cyclic set of edges, then the cycles contained in $s$ are pairwise independent and $s$ is their   (disjoint) union.
\end{lemma}

\begin{proof}   Define a  relation $\sim$ on  $s$  by $e\sim f$ if $e,f\in s$ satisfy $\partial e\cap \partial f\neq \emptyset$. This relation generates an equivalence relation on $s$; the corresponding equivalence classes  are    the cycles contained in $s$.
 This   implies   the lemma.
\end{proof}

   A cycle $s\subset E$ is {\it even} if $s$ has a partition into two   subsets  called the {\it halves}, such that   any two elements of  the same half are carried by the boundary map $\partial$ to disjoint    subsets of~$V$. It is easy to see  that if such a partition  $s=s' \cup  s''$   exists, then it  is unique and $\cup_{e\in s'} \partial e =\cup_{e\in s''} \partial e $.


As in Section~\ref{theglidingsystemcycles},   even cycles   in ${\Gamma}$ in the role of glides together with the
independence relation   above form   a
  regular set-like gliding system in  the power group $G=2^E$.
By Corollary~\ref{simpleNEWBIS++++}, the associated cubed complex $ X_G=X_G({\Gamma})$ (with $0$-skeleton $G $)   is
nonpositively curved. By Section~\ref{orisetthe}, a choice of a distinguished element in each even cycle in ${\Gamma}$  determines an orientation of $G$ and a homomorphism $\mu_A:\pi_1(X_G, A)\to \A(G) $ for all $A\in  G$.
If   $E $ is  finite, then      the right-angled Artin  group    $\A(G) $ is    finitely generated and $\mu_A$ is an injection.

\subsection{Dimer coverings and labelings}\label{dimhyper}
 A {\it dimer covering}  of a  hypergraph $\Gamma=(E,V, \partial)$  is a set   $A\subset E$
   such that each vertex of $\Gamma$  is incident to exactly one  element of $A$. Equivalently, a dimer covering  of  $\Gamma$  is a set $A\subset E$ such that the family $\{\partial e\}_{e\in A}$ is a partition of $V$.
Let ${\D}=\D(\Gamma) $
be the set   of all dimer coverings of $\Gamma$.  The same arguments as in   the proof of Lemma~\ref{dimersMM}  show that
 ${\D}\subset G=2^E$  satisfies   the square condition and the 3-cube condition.
 The associated
cubed complex $X_{\D}  \subset X_G$  with
0-skeleton~${\D}$  is      the {\it dimer complex} of
$\Gamma$.  This complex   is
nonpositively curved.

   A  {\it dimer labeling} of  a  hypergraph $\Gamma=(E,V, \partial)$ is  a
labeling of the edges of $\Gamma$  by non-negative real numbers such
that for every vertex of $\Gamma$, the   labels of the incident
edges
 sum up  to give $1$ and    only one or two of   these labels may be non-zero.    The set $  L(\Gamma)$ of dimer labelings  of   $\Gamma$ is a closed subset of  the cube $I^E$. We endow
$L(\Gamma)$
  with the induced  topology. Each dimer covering   of
$\Gamma$ determines a dimer labeling that carries the edges of
the  covering  to $1$  and  all other edges  to $0$.

\subsection{Finite hypergraphs}\label{fincychyper} A hypergraph $\Gamma=(E,V, \partial)$ is {\it finite} if  the sets   $E$ and~$V$ are finite.  Then the dimer complex $X_\D$ is connected and its  fundamental
group is the {\it dimer group} of $\Gamma$.  With these definitions, the content of Section~\ref{newDimer coverings+} including all statements and arguments
applies {\it verbatim} to finite hypergraphs. One should simply replace the word \lq\lq graph" with \lq\lq hypergraph" everywhere.

                     \end{document}